\documentclass[11pt]{amsart}

\usepackage{amsmath,amsfonts,amsthm,bbm, amssymb}
\usepackage[cmtip, all]{xy}
\usepackage[letterpaper, left=3cm, right=3cm, top = 2.5cm, bottom = 2.5cm ]{geometry}
\usepackage{comment}

\newtheorem{thm}[equation]{Theorem}
\makeatletter
\let\c@subsection\c@equation
\makeatother
\newtheorem{prop}[equation]{Proposition}
\newtheorem{lem}[equation]{Lemma} 
\newtheorem{cor}[equation]{Corollary}

\theoremstyle{definition}

\newtheorem{defn}[equation]{Definition}
\theoremstyle{remark}
\newtheorem{remk}[equation]{Remark}
\newtheorem{remks}[equation]{Remarks}

\newtheorem{exm}[equation]{Example}
\newtheorem{exms}[equation]{Examples}
\newtheorem{notat}[equation]{Notation}
\numberwithin{equation}{section}

{\hfill$\square$\end{defn}}
{\hfill$\square$\end{remk}}
{\hfill$\square$\end{remks}}
{\hfill$\square$\end{exm}}
{\hfill$\square$\end{exms}}
{\hfill$\square$\end{notat}}

\newcommand{\thmref}{Theorem~\ref}

\newcommand{\sB}{{\mathcal B}}
\newcommand{\sC}{{\mathcal C}}
\newcommand{\sD}{{\mathcal D}}

\newcommand{\sF}{{\mathcal F}}
\newcommand{\sG}{{\mathcal G}}
\newcommand{\sH}{{\mathcal H}}
\newcommand{\sI}{{\mathcal I}}

\newcommand{\sK}{{\mathcal K}}
\newcommand{\sL}{{\mathcal L}}
\newcommand{\sM}{{\mathcal M}}

\newcommand{\sO}{{\mathcal O}}

\newcommand{\sR}{{\mathcal R}}

\newcommand{\sZ}{{\mathcal Z}}

\newcommand{\cO}{{\mathcal O}}

\newcommand{\A}{{\mathbb A}}

\newcommand{\C}{{\mathbb C}}

\renewcommand{\H}{{\mathbb H}}

\newcommand{\bL}{{\mathbb L}}

\renewcommand{\P}{{\mathbb P}}
\newcommand{\Q}{{\mathbb Q}}

\newcommand{\Z}{{\mathbb Z}}

\newcommand{\an}{{\rm an}}

\newcommand{\Alb}{{\rm Alb}}
\newcommand{\CH}{{\rm CH}}

\newcommand{\surj}{\twoheadrightarrow}
\newcommand{\inj}{\hookrightarrow}

\newcommand{\Pic}{{\rm Pic}}

\newcommand{\divf}{{\rm div}}
\newcommand{\Hom}{{\rm Hom}}

\newcommand{\Spec}{{\rm Spec \,}}

\newcommand{\Zar}{{\text{\rm Zar}}}

\newcommand{\Sch}{{\operatorname{\mathbf{Sch}}}}

\newcommand{\Sm}{{\mathbf{Sm}}}

\providecommand{\MGL}{\mathop{\rm MGL}\nolimits}
\providecommand{\KGL}{\mathop{\rm KGL}\nolimits}

\newcommand{\Nis}{{\operatorname{Nis}}}

\newcommand{\ds}{{/\kern-3pt/}}

\newcommand{\colim}{\mathop{\text{colim}}}

\newcommand{\ov}{\overline}

\renewcommand{\dim}{\text{\rm dim}}

\newcommand{\tuborg}{\left\{\begin{array}{ll}}
\newcommand{\sluttuborg}{\end{array}\right.}

\newcommand{\wt}{\widetilde}

\newcommand{\fnteff}{\mathrm{eff}}

\newcommand{\unsmot}{\mathcal M}
\newcommand{\unshomotopy}{\mathcal H}
\newcommand{\unsmotcdh}{\mathcal M _{cdh}}
\newcommand{\unshomotopycdh}{\mathcal H _{cdh}}
\newcommand{\Tspectra}{Spt(\unsmot)}
\newcommand{\TspectraX}{Spt(\unsmot_{X})}
\newcommand{\Tspectracdh}{Spt(\unsmotcdh)}

\newcommand{\susp}[1]{\Sigma ^{#1}}
\newcommand{\stablehomotopy}{\mathcal{SH}}
\newcommand{\stablehomotopycdh}{\mathcal{SH}_{cdh}}

\def\cO{\mathcal{O}}

\newcounter{elno}

\newcounter{elno-abc}   

\newcounter{elno-abc-prime}

\begin{document}
\title{The slice spectral sequence for singular schemes and applications}
\author{Amalendu Krishna and Pablo Pelaez}
\address{School of Mathematics, Tata Institute of Fundamental Research,  
1 Homi Bhabha Road, Colaba, Mumbai, India.}
\email{amal@math.tifr.res.in}
\address{Instituto de Matem\'aticas, Ciudad Universitaria, 
UNAM, DF 04510, M\'exico.}
\email{pablo.pelaez@im.unam.mx}



\begin{abstract}
We examine the slice spectral sequence for the cohomology of singular 
schemes with respect to various motivic $T$-spectra, especially the
motivic cobordism spectrum. When the base field $k$ admits resolution of
singularities and $X$ is a scheme of finite type over $k$, 
we show that Voevodsky's slice filtration leads to a spectral sequence for
$\MGL_X$ whose terms are the motivic cohomology groups of $X$ defined
using the $cdh$-hypercohomology. 
As a consequence, we establish an isomorphism between certain geometric
parts of the motivic cobordism and motivic cohomology of $X$.

A similar spectral sequence for the connective $K$-theory leads to
a cycle class map from the motivic cohomology to the homotopy invariant
$K$-theory of $X$. We show that this cycle class map is injective for
a large class of projective schemes. 
We also deduce applications to the torsion in the
motivic cohomology of singular schemes. 
\end{abstract} 
\setcounter{tocdepth}{1}
\maketitle
  
\tableofcontents

\section{Introduction}\label{sec:Intro}
The motivic homotopy theory of schemes was put on a firm foundation by
Voevodsky and his coauthors beginning with the work of Morel and
Voevodsky \cite{MV} and its stable counterpart \cite{Voev-0}.
It was observed by Voevodsky \cite{Voev-2} that the motivic $T$-spectra in 
the stable homotopy category $\stablehomotopy_X$ over a noetherian scheme $X$ 
of finite Krull dimension can be understood via their slice filtration.
This slice filtration leads to spectral sequences which then become a
very powerful tool in computing various cohomology theories
for smooth schemes over $X$.

The main problem in the study of the slice filtration for a given motivic $T$-spectrum
is twofold: the identification of its  slices and the analysis of the convergence
properties for the corresponding slice spectral sequence.
When $k$ is a field which admits resolution of singularities, 
the slices for many of these motivic $T$-spectra
in $\stablehomotopy_k$ are now known. 
In particular, we can compute these generalized cohomology groups
of smooth schemes over $k$ using the slice spectral sequence.

In this paper, we study some descent property of the motivic $T$-spectra
in $\stablehomotopy_X$ when $X$ is a possibly singular scheme of
finite type over $k$. This descent property tells us that the cohomology
groups of a scheme $Y \in \Sm_X$, associated to an absolute
motivic $T$-spectra in $\stablehomotopy_X$ \cite[\S 1.2]{Deglise},
can be computed using only $\stablehomotopy_k$.

Even though our methods apply to any of these absolute $T$-spectra, we
restrict our study to the motivic cobordism spectrum $\MGL_X$.
We show using the above descent property of motivic spectra that
$\MGL_X$ can be computed using the motivic cohomology groups of
$X$. Recall from \cite[Defs.~4.3, 9.2]{FV} that the motivic 
cohomology groups of $X$ are defined to be the $cdh$-hypercohomology groups
$H^p(X, \Z(q)) = \H^{p-2q}_{cdh}(X,  C_{\ast}z_{equi}(\A^q_k,0)_{cdh})$.
Using these motivic cohomology groups, we show:

\begin{thm}\label{thm:Main-1}
Let $k$ be a field which admits resolution of singularities
and let $X$ be a separated scheme of finite type over $k$.
Then for any integer $n \in \Z$,
there is a strongly convergent spectral sequence
\begin{equation}\label{eqn:Main-1-0}
E^{p,q}_2 = H^{p-q}(X, \Z(n-q)) \otimes_{\Z} \bL^{q}
\Rightarrow \MGL^{p+q, n} (X)
\end{equation}
and the differentials of this spectral sequence are given by
$d_r: E^{p,q}_r \to E^{p+r, q-r+1}_r$.
Furthermore, this spectral sequence degenerates with rational coefficients.
\end{thm}

If $k$ is a perfect field of positive characteristic $p$, we obtain a
similar spectral sequence after inverting $p$ except that we can not guarantee
strong convergence unless $X$ is smooth over $k$ 
(see ~\ref{rmk.MGLnoconv}).

As a consequence of ~\ref{thm:Main-1} and its positive characteristic
version, we get the following relation between  the motivic 
cobordism and cohomology of singular schemes.

\begin{thm}\label{thm:Main-2}
Let $k$ be a field which admits resolution of singularities (resp. a perfect field 
of positive characteristic $p$). Then
for any separated (resp. smooth) scheme $X$ of finite type over $k$ and dimension $d$
and every $i \ge 0$, the edge map in the spectral sequence
\eqref{eqn:Main-1-0}:
\begin{align*}
\nu_X: \MGL^{2d+i,d+i}(X)  &
\to H^{2d+i}(X, \Z(d+i))  \\ resp. \; \;
\nu_X: \MGL^{2d+i,d+i}(X) {\underset{\Z}\otimes} \Z[\tfrac{1}{p}] &
\to H^{2d+i}(X, \Z(d+i)) {\underset{\Z}\otimes} \Z[\tfrac{1}{p}] 
\end{align*}
is an isomorphism.
\end{thm}

We apply our descent result to obtain a similar spectral sequence
for the connective $KH$-theory, $\KGL^0$ (see \S \ref{sec:Appl}). We use this
spectral sequence and the canonical map $CKH(-) \to KH(-)$
from the connective $KH$-theory to obtain the following cycle 
class map from the motivic cohomology of a singular scheme to its homotopy 
invariant $K$-theory.

\begin{thm}\label{thm:Main-3}
Let $k$ be a field of exponential characteristic $p$ and let
$X$ be a separated scheme of dimension $d$ which is of finite type over $k$.
Then the map $\KGL^0_X \to s_0\KGL_X \cong H\Z$ induces for every integer 
$i \ge 0$, an isomorphism
\[
CKH^{2d+i, d+i}(X) {\underset{\Z}\otimes} \Z[\tfrac{1}{p}] 
\xrightarrow{\cong} H^{2d+i}(X, \Z(d+i)) 
{\underset{\Z}\otimes} \Z[\tfrac{1}{p}].
\]

In particular, there is a natural cycle class map
\[
cyc_i: H^{2d+i}(X, \Z(d+i)) {\underset{\Z}\otimes} \Z[\tfrac{1}{p}] \to 
KH_i(X) {\underset{\Z}\otimes} \Z[\tfrac{1}{p}].
\]
\end{thm}

We use this cycle class map and the Chern class maps from the homotopy
invariant $K$-theory to the Deligne cohomology of schemes over $\C$ to
construct intermediate Jacobians and Abel-Jacobi maps for the motivic
cohomology of singular schemes over $\C$. More precisely, we prove the
following. This generalizes intermediate Jacobians and
Abel-Jacobi maps of Griffiths and the torsion theorem of Roitman for smooth 
schemes.

\begin{thm}\label{thm:Main-4}
Let $X$ be a projective scheme over $\C$ of dimension $d$. 
Assume that either $d \le 2$ or $X$ is regular in codimension one.
Then, there is 
a semi-abelian variety $J^d(X)$ and an Abel-Jacobi map 
${\rm AJ}^d_X: H^{2d}(X, \Z(d))_{\deg 0}  \to J^d(X)$ which is surjective 
and whose restriction to the torsion subgroups is an isomorphism.
\end{thm}

In a related work, Kohrita (see \cite[Corollary~6.5]{Kor}) has constructed an 
Abel-Jacobi map for the Lichtenbaum motivic cohomology 
$H^{2d}_L(X, \Z(d))$ of singular
schemes over $\C$ using a different technique. He has also
proven a version of Roitman torsion theorem for the 
Lichtenbaum motivic cohomology.
The natural map $H^{2d}(X, \Z(d)) \to
H^{2d}_L(X, \Z(d))$ is not an isomorphism in general if $d \ge 3$.
Note also that the Roitman torsion theorem for $H^{2d}(X, \Z(d))$
is a priori a finer statement than that for the analogous Lichtenbaum cohomology.

\vskip .3cm

Using \thmref{thm:Main-4}, 
we prove the following property of the cycle class map
of ~\ref{thm:Main-3} which is our final result. 
The analogous result for smooth projective schemes
was proven by Marc Levine \cite[Thm. 3.2]{Levine}.
More generally, Levine shows that a relative Chow group of 0-cycles on
a normal projective scheme over $\C$ injects inside $K_0(X)$.

\begin{thm}\label{thm:Main-5}
Let $X$ be a projective scheme of dimension $d$ over $\C$.
Assume that either $d \le 2$ or $X$ is regular in codimension one.
Then the cycle class map $cyc_0: H^{2d}(X, \Z(d)) \to KH_0(X)$ is injective.
\end{thm}

We end this section with the comment that our motivation behind this work was 
to exploit powerful tools of the
motivic homotopy theory to study several questions about the
motivic cohomology and $K$-theory of singular schemes, which were
previously known only for smooth schemes. We hope that the methods and the
techniques of 
our proofs can be advanced further to answer many other 
cohomological questions about singular schemes. 
We refer to \cite{KP} for more results based on the techniques of this
text.


\section{A descent theorem for motivic spectra}\label{sec:Prelim}
In this section, we set up our notations, discuss various model
structures we use in our proofs and show the Quillen adjunction property
of many functors among these model structures.
The main objective of this section is to prove a $cdh$-descent property
of the motivic $T$-spectra, see  ~\ref{thm:basch2}.

\subsection{Notations and preliminary results}\label{sec:Notn}
Let $k$ be a perfect field of exponential characteristic $p$ (in
some instances we will require that the field $k$ admits resolution of singularities 
\cite[Def. 4.1]{Voev-5}).  We will
write $\Sch_k$ for the category of separated schemes of finite type over $k$ 
and $\Sm_k$ for the full subcategory of $\Sch_k$ consisting of smooth schemes 
over $k$.  If $X\in \Sch_k$, let $\Sm_X$ denote the full subcategory of 
$\Sch_k$ consisting of 
smooth schemes over $X$.  We will write $(\Sm_k)_{Nis}$ (resp. $(\Sm_X)_{Nis}$, 
$(\Sch_k)_{cdh}$, $(\Sch_k)_{Nis}$) for $\Sm_k$ equipped with the Nisnevich topology (resp. 
$\Sm_X$ equipped with the Nisnevich topology, $\Sch_k$ equipped with the
$cdh$-topology, $\Sch _k$ equipped with the Nisnevich topology).
The product $X \times _{\Spec{k}} Y$ will be denoted by $X \times Y$.

Let $\unsmot$ (resp. $\unsmot _{X}$, $\unsmotcdh$) be the category of pointed 
simplicial presheaves on $\Sm_k$ (resp. $\Sm_X$, $\Sch_k$) equipped with the  
motivic model structure described in \cite{Isak}
considering the Nisnevich topology on $\Sm_k$ (resp. Nisnevich topology on
$\Sm_X$, $cdh$-topology on $\Sch_k$) and the affine line $\A^1_{k}$ as an 
interval. A simplicial presheaf will often be called a {\sl motivic space}.

We define $T$ in $\unsmot$ (resp. $\unsmot _{X}$, $\unsmotcdh$) as the pointed 
simplicial presheaf
represented by $S^{1}_s \wedge S^1_t$, where $S^1_t$ is 
$\A^1_{k} \setminus  \{ 0 \}$ (resp. $\A^1_{X} \setminus \{ 0\}$, 
$\A^1_{k} \setminus \{ 0\}$) pointed by $1$, and $S^{1}_s$ denotes the 
simplicial circle.  
Given an arbitrary integer $r \geq1$, let $S^{r}_s$
(resp. $S^{r}_t$) denote the iterated smash product of $S^{1}_s$ (resp. $S^1_t$) 
with $r$-factors: $S^{1}_s\wedge \cdots \wedge S^{1}_s$ 
(resp. $S^1_t \wedge \cdots
\wedge S^1_t$); $S^{0}_{s} = S^{0}_{t}$ will be by definition equal to the 
pointed simplicial presheaf represented by the base scheme $\Spec{k}$ 
(resp. $X$, $\Spec{k}$). 
 
Since $T$ is cofibrant in $\unsmot$ (resp. $\unsmot _{X}$, $\unsmotcdh$)
we can apply freely the results in \cite[\S 8]{Hovey}.  Let
$\Tspectra$ (resp. $\TspectraX$, $\Tspectracdh$) denote the category of 
symmetric $T$-spectra on $\unsmot$ (resp. $\unsmot _{X}$, $\unsmotcdh$) 
equipped with 
the motivic model structure defined in \cite[8.7]{Hovey}.
We will write $\stablehomotopy$ (resp. 
$\stablehomotopy _{X}$,
$\stablehomotopycdh$) for the homotopy category of $\Tspectra$ (resp. 
$\TspectraX$, $\Tspectracdh$)
which is a tensor triangulated category.
 For any two integers $m$, $n \in \Z$, let $\Sigma^{m,n}$ denote the
automorphism $\Sigma^{m-n}_s \circ \Sigma^n_t: \stablehomotopy \to
\stablehomotopy$ (this also makes sense in $\stablehomotopy _{X}$ and $\stablehomotopycdh$).
We will write $\Sigma _{T}^{n}$ for $\Sigma ^{2n,n}$, and
$E\wedge F$ for the smash product of $E$, $F\in \stablehomotopy$ (resp. 
$\stablehomotopy _{X}$, $\stablehomotopycdh$).

Given a simplicial presheaf $A$, we will write $A_{+}$ for the pointed 
simplicial presheaf
obtained by adding a disjoint base point (isomorphic to the base scheme) to 
$A$.  For any $B \in \unsmot$, let $\Sigma^{\infty}_T(B)$ denote the
object $(B, T \wedge B, \cdots ) \in \Tspectra$.
This functor makes sense for objects in $\unsmotcdh$ and $\unsmot _{X}$
as well.

If $F:\mathcal A\rightarrow \mathcal B$ is a
functor with right adjoint $G:\mathcal B \rightarrow \mathcal A$, 
we shall say that $(F,G):
\mathcal A \rightarrow \mathcal B$ is an adjunction.  We shall freely use the 
language of
model and triangulated categories.  We will write $\susp{1}$ for the 
suspension functor
in a triangulated category, and for $n\geq 0$ (resp. $n<0$), $\susp{n}$ 
will be the suspension (resp. desuspension) functor iterated $n$ (resp. $-n$) 
times.

We will use the following notation in all the categories under 
consideration: $\ast$ 
will denote the terminal object, and $\cong$ will denote that a map 
(resp. functor) is an isomorphism (resp. equivalence of categories).

\subsection{Change of site}\label{sec:sitech}
Let $X\in \Sch_k$ and let $v: X \to \Spec k$ denote the structure map.
We will write $Pre _{X}$ (resp. $\underline{Pre} _{k}$) for 
the category of pointed simplicial
presheaves on $\Sm_X$ (resp.  $\Sch_k$). If $X=\Spec{k}$
where $k$ is the base field, we will write $Pre _{k}$ instead of $Pre _{X}$. 
These categories are equipped with the objectwise
flasque model structure \cite[\S 3]{Isak}. To recall this model structure,
we consider a finite set $I$ of   
monomorphisms $\{V_i \to U\}_{i \in I}$ for any $U \in \Sm_X$.
The categorical union 
${\underset{i \in I}\cup} V_i$ is the coequalizer of the diagram
\[
\xymatrix{
{\underset{i,j \in I}\coprod} \ V_i {\underset{U}\times} \ V_j
\ar@<+.7ex>[r] 
\ar@<-.7ex>[r] &
{\underset{i \in I}\coprod} \ V_i
}
\]
formed in $Pre _X$. We denote by $i_I$ the induced monomorphism
${\underset{i \in I}\cup} V_i \to U$. Note that $\emptyset \to U$ arises in this
way. The push-out product of maps of $i_I$ and a map between simplicial sets
exists in $Pre_X$. In particular, we are entitled to form the sets
\[
I^{\rm sch}_{\rm clo}(\Sm_X) = 
\{i_I  \ \square \ {(\partial \Delta^n \subset \Delta^n)}_{+}\}_{I, n \ge 0}
\]
and
\[
J^{\rm sch}_{\rm clo}(\Sm_X) = 
\{i_I \ \square \ (\Lambda^n_i \subset \Delta^n)_{+}\}_{I, n \ge 0, 0 \le i \le n},
\]
where $I$ is a finite set of monomorphisms $\{V_i \to U\}_{i \in I}$, $U\in \Sm _X$;
and $i_I : {\underset{i \in I}\cup} V_i \to U$ is the induced monomorphism
defined above.

A map between simplicial presheaves is called a closed objectwise fibration if
it has the right lifting property with respect to $J^{\rm sch}_{\rm clo}(\Sm_X)$.
A map $u: E \to F$ between simplicial presheaves is called a weak equivalence if
$E(U) \to F(U)$ is a weak equivalence of simplicial sets for each
$U \in \Sm_X$. A closed objectwise cofibration is a map having the left 
lifting property with respect to every trivial closed objectwise fibration.
Note that this notion of weak equivalence, cofibrations and 
fibrations makes sense for simplicial presheaves in=
any category with finite products (e.g., $\Sm_k$, $\Sch_k$).
It follows from \cite[Thm. 3.7]{Isak} that the above notion of
weak equivalence, cofibrations and fibration forms a proper, simplicial 
and cellular model category structure on $Pre_k$, $Pre_X$
and  $\underline{Pre} _{k}$. We call this the {\sl objectwise flasque model
structure}. Our reason for choosing this model structure is the 
following result.

\begin{lem}$($\cite[Lem. 6.2]{Isak}$)$\label{lem:cofibrant-fl}
If $V \to U$ is a monomorphism in $\Sm_k$ (resp. $\Sm_X$, $\Sch_k$),
then ${U_+}/{V_+}$ is cofibrant in the flasque model structure on 
$Pre _k$ (resp. $Pre _{X}$, $\underline{Pre} _k$). In particular,
$T^n \wedge U_{+}$ is cofibrant for any $n \ge 0$.
\end{lem}

It is clear that $Pre_{X}$ (resp. $\underline{Pre} _{k}$) is a cofibrantly 
generated model category with generating cofibrations 
$I^{\rm sch}_{\rm clo}(\Sm_X)$ (resp. $I^{\rm sch}_{\rm clo}(\Sch_k)$)
and generating trivial cofibrations $J^{\rm sch}_{\rm clo}(\Sm_X)$
(resp. $J^{\rm sch}_{\rm clo}(\Sch_k)$).

Let $\pi: (\Sch_k)_{cdh}\rightarrow (\Sm_k)_{Nis}$ be the continuous map of sites
considered by Voevodsky in \cite[\S 4]{Voev-5}. 
We will write $(\pi ^{\ast}, \pi _{\ast}):Pre _{k} \rightarrow 
\underline{Pre} _{k}$,
$(v ^{\ast}, v _{\ast}):Pre _{k} \rightarrow Pre _{X}$ for the adjunctions 
induced by $\pi$, $v$ respectively.

We will also consider the morphism of sites 
$\pi _{X}: (\Sch_k)_{cdh}\rightarrow (\Sm_X)_{Nis}$
and the corresponding adjunction $(\pi_{X} ^{\ast}, \pi_{X \ast}):Pre _{X} 
\rightarrow \underline{Pre} _{k}$.
These adjunctions are related by the following.

\begin{lem}\label{lem:commdiag1}
The following diagram commutes:
\[  
\xymatrix@C1pc{
Pre_{k} \ar[r]^-{\pi ^{\ast}} \ar[dr]_-{v^{\ast}}& \underline{Pre}_{k} \ar[d]^-
{\pi _{X\ast}}\\
		& Pre_{X}.}
\]
\end{lem}
\begin{proof}
We first notice that for every simplicial set $K$, 
$Y\in \Sm_k$ and $Z\in \Sm_X$, one has
\begin{equation}\label{eqn:commdiag1-0}
\pi ^{\ast}(K\otimes Y_{+})=K\otimes Y_{+} \in \underline{Pre}_{k};
\ \ \
v^{\ast}(K\otimes Y_{+})=K\otimes (Y\times X)_{+} \in Pre_{X} \
\mbox{and} 
\end{equation}
\[
\pi _{X}^{\ast}(K\otimes Z_{+})=K\otimes Z_{+} \in \underline{Pre}_{k}.
\]

We observe that $\pi^{\ast}$, $v^{\ast}$ commute with colimits since they are 
left adjoint; and that
$\pi _{X\ast}$ also commutes with colimits since it is a restriction functor.  
Hence, it suffices to show
that for every simplicial set $K$ and every 
$Y\in \Sm_k$, $\pi _{X\ast}(\pi ^{\ast}(K\otimes Y_{+}))=
v^{\ast}(K\otimes Y_{+})$.  Finally, a direct computation 
shows that $\pi _{X\ast}(K\otimes Y_{+})=K\otimes (Y\times X)_{+} \in 
Pre_{X}$ and we conclude by ~\eqref{eqn:commdiag1-0}.
\end{proof}

\begin{lem}\label{lem:prunsbasech}
The adjunctions $(\pi ^{\ast}, \pi _{\ast}):Pre _{k} \rightarrow 
\underline{Pre} _{k}$, $(v^{\ast}, v _{\ast}):Pre _{k} \rightarrow Pre_{X}$ and
$(\pi _{X} ^{\ast}, \pi _{X\ast}):Pre _{X} \rightarrow \underline{Pre} _{k}$ are 
all Quillen adjunctions. Moreover, $\pi _{X\ast}$ and $\pi_*$ preserve weak 
equivalences.
\end{lem}
\begin{proof}
We have seen above that all the three model categories (with the objectwise
flasque model structure) are cofibrantly generated. Moreover, it follows from 
~\eqref{eqn:commdiag1-0} that $\pi^{\ast}(I^{\rm sch}_{\rm clo}(\Sm_k))$ \\
$\subseteq I^{\rm sch}_{\rm clo}(\Sch_k)$, 
$\pi^{\ast}(J^{\rm sch}_{\rm clo}(\Sm_k))\subseteq J^{\rm sch}_{\rm clo}(\Sch_k)$;
$v^{\ast}(I^{\rm sch}_{\rm clo}(\Sm_k)) \subseteq J^{\rm sch}_{\rm clo}(\Sm_X)$, 
$v^{\ast}(J^{\rm sch}_{\rm clo}(\Sm_k)) \subseteq J^{\rm sch}_{\rm clo}(\Sm_X)$ and
$\pi _{X}^{\ast}(I^{\rm sch}_{\rm clo}(\Sm_X)) \subseteq 
I^{\rm sch}_{\rm clo}(\Sch_k)$, $\pi _{X}^{\ast}(J^{\rm sch}_{\rm clo}(\Sm_X))
\subseteq J^{\rm sch}_{\rm clo}(\Sch_k)$.  
Hence, it follows from \cite[Lem. 2.1.20]{Hovey-1} that
$(\pi ^{\ast}, \pi _{\ast})$, $(v^{\ast}, v_{\ast})$ and 
$(\pi _{X} ^{\ast}, \pi _{X\ast})$ are Quillen adjunctions.
The second part of the lemma is an immediate consequence of the 
observation that $\pi _{X\ast}$ and $\pi_*$ are restriction functors and the weak
equivalences in the objectwise flasque model structure are defined schemewise.
\end{proof}

To show that the Quillen adjunction of \ref{lem:prunsbasech} extends to the
level of motivic model structures, we consider a distinguished square 
$\alpha$ (see \cite[\S 2]{Voev-5}):
\begin{equation}\label{eqn:Dist-sq} 
  \begin{array}{c}
\xymatrix@C1pc{
Z' \ar[r] \ar[d]& Y' \ar[d]\\
		Z \ar[r]& Y}
  \end{array}
\end{equation}
in $(\Sm_k)_{Nis}$ (resp. $(\Sm_X)_{Nis}$, $(\Sch_k)_{cdh}$).  We will write
$P(\alpha)$ for the pushout of $Z \leftarrow Z'\rightarrow Y'$ in $Pre _{k}$ 
(resp. $Pre_{X}$, $\underline{Pre}_{k}$).

The motivic model category $\unsmot$ (resp. $\unsmot _{X}$, $\unsmotcdh$, 
$\unsmot _{ft}$) is 
the left Bousfield localization of $Pre _{k}$ 
(resp. $Pre_{X}$, $\underline{Pre}_{k}$, $\underline{Pre}_{k}$) with respect
to the following two sets of maps: 
\begin{itemize}
\item $P(\alpha)\rightarrow Y$ indexed by the distinguished 
squares in $(\Sm_k)_{Nis}$ (resp. $(\Sm_X)_{Nis}$, $(\Sch_k)_{cdh}$,
$(\Sch_k)_{Nis}$),
\item $p_{Y}:Y\times \A^1 _{k}\rightarrow Y$ for
$Y\in \Sm_k$ (resp. $Y\in \Sm_X$, $Y \in \Sch_k$, $Y \in \Sch_k$).
\end{itemize}
Notice that since we are working with the
flasque model structures, by \cite[Thms. 4.8-4.9]{Isak}
it is possible to consider maps from the ordinary pushout
$P(\alpha)$ instead of maps from the homotopy pushout of the diagram
$Z\leftarrow  Z' \rightarrow Y'$ \eqref{eqn:Dist-sq}.

\begin{remk}  \label{rmk.locmodstr}
We will also consider the Nisnevich (resp. $cdh$) local
model structure, i.e., the
left Bousfield localization of
$Pre _{k}$ (resp. $\underline{Pre}_{k}$) with respect
to the set of maps: 
$P(\alpha)\rightarrow Y$ indexed by the distinguished 
squares in $(\Sm_k)_{Nis}$ (resp. $(\Sch_k)_{cdh}$).
\end{remk}

We will abuse notation and write 
$(\pi ^{\ast}, \pi _{\ast}):\unsmot \rightarrow \unsmotcdh$,
$(v^{\ast}, v _{\ast}):\unsmot \rightarrow \unsmot _{X}$,
$(\pi _{X}^{\ast}, \pi _{X\ast}):\unsmot _{X} \rightarrow \unsmotcdh$ for the 
adjunctions
induced by $\pi$, $v$ and $\pi _{X}$, respectively.

\begin{prop}\label{prop:unsbasech}
The adjunctions $(\pi ^{\ast}, \pi _{\ast}):\unsmot \rightarrow \unsmotcdh$,
$(v^{\ast}, v _{\ast}):\unsmot \rightarrow \unsmot _{X}$,
$(\pi _{X}^{\ast}, \pi _{X\ast}):\unsmot _{X} \rightarrow \unsmotcdh$ are 
Quillen adjunctions.
\end{prop}
\begin{proof}
We will give the argument for $(\pi ^{\ast}, \pi _{\ast})$, since the other 
cases are parallel.
Consider the commutative diagram  
\[	
\xymatrix@C1pc{
Pre_{k} \ar[r]^-{\pi^{\ast}} \ar[d]_-{id}& \underline{Pre}_{k} \ar[d]^-{id}\\
		\unsmot \ar@{-->}[r]^-{\pi ^{\ast}}& \unsmotcdh,}
\]
where the solid arrows are left Quillen functors by 
\cite[Lem. 3.3.4(1)]{Hirsc} and \ref{lem:prunsbasech}.  
Thus, it follows from \cite[Def. 3.1.1(1)(b), Thm. 3.3.19]{Hirsc} 
that it suffices to check that $\pi^{\ast}(P(\alpha) \rightarrow Y)$ and
$\pi^{\ast}(Y\times \A^1 _{k} \rightarrow Y)$ are weak equivalences in 
$\unsmotcdh$.

On the one hand, it is immediate that
$\pi^{\ast}(Y\times \A^1 _{k} \rightarrow Y)= (Y\times \A^1_{k}\rightarrow Y) \in
\unsmotcdh$; hence a weak equivalence in $\unsmotcdh$.
On the other hand, $\pi^{\ast}$ commutes with pushouts since it is a left 
adjoint functor. It follows therefore from ~\eqref{eqn:commdiag1-0}
that $\pi^{\ast}(P(\alpha)\rightarrow Y)= (P(\alpha)\rightarrow Y) \in
\unsmotcdh$, hence a weak equivalence in $\unsmotcdh$.
\end{proof}
We will write $\unshomotopy$ (resp. $\unshomotopy _{X}$, $\unshomotopycdh$) for
the homotopy category of $\unsmot$ (resp. $\unsmot _{X}$, $\unsmotcdh$) and
$(\mathbf L \pi ^{\ast}, \mathbf R \pi _{\ast}):\unshomotopy \rightarrow 
\unshomotopycdh$,
$(\mathbf L v^{\ast}, \mathbf R v _{\ast}):\unshomotopy \rightarrow 
\unshomotopy _{X}$,
$(\mathbf L \pi _{X}^{\ast}, \mathbf R \pi _{X\ast}):\unshomotopy _{X} \rightarrow 
\unshomotopycdh$ for the derived adjunctions of the Quillen adjunctions in
~\ref{prop:unsbasech} (see \cite[Thm. 3.3.20]{Hirsc}).

\subsection{A $cdh$-descent for motivic spectra}\label{sec:M-Desc}
It follows from ~\eqref{eqn:commdiag1-0} that the adjunctions 
between the categories of motivic spaces induce levelwise
adjunctions $(\pi ^{\ast}, \pi _{\ast}):\Tspectra \rightarrow \Tspectracdh$,
$(v ^{\ast}, v _{\ast}):\Tspectra \rightarrow \TspectraX$,
$(\pi _{X}^{\ast}, \pi _{X\ast}):\TspectraX \rightarrow \Tspectracdh$ between the
corresponding categories of symmetric $T$-spectra such that the following
diagram commutes (see \eqref{lem:commdiag1}):
\begin{equation}\label{eqn:commdiag2} 
\begin{array}{c}
	\xymatrix{\Tspectra \ar[r]^-{\pi ^{\ast}} \ar[dr]_-{v^{\ast}}& 
\Tspectracdh 
	\ar[d]^-{\pi _{X\ast}}\\ & \TspectraX .}
\end{array}
\end{equation}

We further conclude from ~\ref{prop:unsbasech}
 and \cite[Thm. 9.3]{Hovey} the following:

\begin{prop}\label{prop:tablebasech}
The pairs 
\begin{enumerate}
\item
$(\pi ^{\ast}, \pi _{\ast}): \Tspectra \to \Tspectracdh$,
\item
$(v ^{\ast}, v _{\ast}):  \Tspectra \to \TspectraX$ and
\item
$(\pi _{X}^{\ast}, \pi _{X\ast}): \TspectraX \to \Tspectracdh$
\end{enumerate}
are Quillen adjunctions between stable model categories.
\end{prop}

We deduce from ~\ref{prop:tablebasech} that there are pairs of
adjoint functors $(\mathbf L \pi ^{\ast}, \mathbf R \pi _{\ast}):\stablehomotopy 
\rightarrow \stablehomotopycdh$,
$(\mathbf L v^{\ast}, \mathbf R v _{\ast}):\stablehomotopy \rightarrow 
\stablehomotopy _{X}$ and
$(\mathbf L \pi _{X}^{\ast}, \mathbf R \pi _{X\ast}):\stablehomotopy _{X} 
\rightarrow \stablehomotopycdh$ between the various stable homotopy
categories of motivic $T$-spectra.  
We observe that for $a\geq b\geq 0$, the suspension functor $\Sigma^{a,b}$
in $\stablehomotopy$ (resp. $\stablehomotopy _X$, $\stablehomotopycdh$)
is the derived functor of the left Quillen functor $E\mapsto S^{a-b}_s\wedge S^b_t\wedge E$
in $\Tspectra$ (resp. $\TspectraX$, $\Tspectracdh$).
Since the functors $\pi ^{\ast}$, $v^{\ast}$, $\pi _{X}^{\ast}$
are simplicial and symmetric monoidal, we deduce that they commute 
with the suspension functors $\Sigma^{m,n}$, i.e.,
for every $m$, $n\in \mathbb Z$:
$ \mathbf L \pi ^{\ast} \circ \Sigma^{m,n} (-) \cong
\Sigma^{m,n} \circ \mathbf L \pi ^{\ast} (-)$,
$\mathbf L v ^{\ast} \circ \Sigma^{m,n} (-) \cong
\Sigma^{m,n} \circ \mathbf L v ^{\ast} (-)$ and
$ \mathbf L \pi ^{\ast}_{X} \circ \Sigma^{m,n} (-) \cong
\Sigma^{m,n} \circ \mathbf L \pi ^{\ast}_{X} (-)$.

Recall that $\unsmot _{ft}$ is the motivic category for the Nisnevich topology
in $\Sch _k$.  We will write $Spt(\unsmot _{ft})$ for the category of symmetric
$T$-spectra on $\unsmot _{ft}$ equipped with the stable model structure
considered in \cite[8.7]{Hovey}.

It is well known \cite[p. 198]{Jardine-2} that $Spt(\unsmot _{ft})$ (resp. $\TspectraX$, $X\in \Sch_k$) 
is a simplicial model category
\cite[9.1.6]{Hirsc}.  For $E$, $E' \in Spt(\unsmot _{ft})$ (resp. $\TspectraX$), we will write $Map(E,E')$ 
(resp. $Map_X (E,E')$)
for the simplicial set of maps from $E$ to $E'$, i.e., the simplicial set with
$n$-simplices of the form $\Hom _{Spt(\unsmot _{ft})}(E\otimes \Delta ^n, E')$
(resp. $\Hom _{\TspectraX}(E\otimes \Delta ^n, E')$).

Given $f:X\rightarrow X'$, we observe that the Quillen adjunction $(f^\ast, f_\ast):Spt(\unsmot _{X'})\rightarrow
\TspectraX$ \cite[Thm. 4.5.14]{Ay-2} is enriched on simplicial sets, i.e.
$Map_{X}(f^\ast E', E)\cong Map_{X'}(E', f_\ast E)$ for $E\in \TspectraX$, $E' \in Spt(\unsmot _{X'})$.

The following result is a direct consequence of the proper base change theorem in
motivic homotopy theory \cite[Cor. 1.7.18]{Ay}, \cite[2.3.11(2)]{CisDegtr-1},
\cite[Prop. 3.7]{Cisinski}.

\begin{prop}\label{prop:basechfor}
The functor $\mathbf L v ^{\ast}$
is naturally equivalent to the composition 
$\mathbf R \pi _{X\ast}\circ \mathbf L \pi ^{\ast}$.
\end{prop}
\begin{proof}
We observe that the following diagram of left Quillen functors commutes:
\[  \xymatrix{\Tspectra \ar[r]^-{\pi ^{\ast}} \ar[dr]_-{\pi_{ft}^{\ast}}& 
\Tspectracdh 
	\\ & Spt(\unsmot _{ft}) . \ar[u]_-{id}}
\]
Let $E$ be a motivic $T$-spectrum in $\Tspectra$.  Without any loss of 
generality, we can assume that $E$ is cofibrant in $\Tspectra$.  
Let $\nu: \pi_{ft} ^{\ast}E\rightarrow E'$ be a functorial fibrant replacement of 
$\pi _{ft}^{\ast}E$ in $Spt(\unsmot _{ft})$.

The argument of Jardine in \cite[pp. 198-199]{Jardine-2} shows that the 
restriction functor $\pi _{X\ast}$ maps weak equivalences in $Spt (\unsmot _{ft})$ into
weak equivalences in $\TspectraX$. Combining this with ~\eqref{eqn:commdiag2},
we  deduce that $\pi _{X\ast}(\nu):\pi _{X\ast}(\pi _{ft}^{\ast}E)= \pi _{X\ast}(\pi ^{\ast}E)=
v^{\ast}E\rightarrow \pi _{X\ast}E'$ is a weak equivalence in $\TspectraX$.  
Since $E$ is cofibrant in $\Tspectra$, $\mathbf L v ^{\ast} E\cong v^\ast E$.
Hence, in order to conclude it suffices to show that $E'$ is fibrant in $\Tspectracdh$.

We shall use the following notation in the rest of the proof: for 
$Y\in \Sch _k$, we will
write $v_Y:Y\rightarrow \Spec (k)$ for the structure map.  Notice that we have proved
that $\mathbf L v_Y ^\ast E \cong v_Y^\ast E \cong \pi _{Y\ast}E'$ in $\stablehomotopy _Y$.
Consider a distinguished abstract
blow-up square in $\Sch _k$, i.e., a distinguished square in the lower cd-structure
defined in \cite[\S 2]{Voev-5}:
\begin{align*}
  \begin{array}{c}
\xymatrix@C1pc{
Z' \ar[r]^-{i'} \ar[d]_-{f'}& Y' \ar[d]^-f\\
		Z \ar[r]_-i& Y}
  \end{array}
\end{align*}
Let $j=i\circ f'$.  Then $\mathbf R f_\ast \mathbf L f ^\ast(\mathbf L v_Y ^\ast E)\cong
\mathbf R f_\ast \mathbf L (v_Y \circ f) ^\ast E\cong \mathbf R f_\ast \pi _{Y' \ast}E'
\cong f_\ast \pi _{Y' \ast}E'$ in $\stablehomotopy _Y$; where the last isomorphism
follows from the fact that $\pi _{Y' \ast}E'$ is fibrant in $Spt (\unsmot _{Y'})$,
since $E'$ is fibrant in $Spt(\unsmot _{ft})$ and the restriction functor
$\pi _{Y'}:Spt(\unsmot _{ft}) \rightarrow Spt (\unsmot _{Y'})$ is a right Quillen functor
(using the same argument as in \ref{prop:tablebasech}).  Similarly, we conclude that
$\mathbf R i_\ast \mathbf L i^\ast(\mathbf L v_Y ^\ast E)\cong i_\ast \pi _{Z \ast}E'$
and $\mathbf R j_\ast \mathbf L j^\ast(\mathbf L v_Y ^\ast E)\cong j_\ast \pi _{Z' \ast}E'$
in $\stablehomotopy _Y$.

Thus, by \cite[Prop. 3.7]{Cisinski} we conclude that the commutative diagram:
\begin{align*}
  \begin{array}{c}
\xymatrix@C1pc{
\pi _{Y\ast}E' \ar[r] \ar[d]&  f_\ast \pi _{Y' \ast}E' \ar[d]\\
		i_\ast \pi _{Z \ast}E' \ar[r]& j_\ast \pi _{Z' \ast}E'}
  \end{array}
\end{align*}
is a homotopy cofiber square in $Spt (\unsmot _Y)$ \cite[13.5.8]{Hirsc}, thus also
a homotopy fibre square since $Spt (\unsmot _Y)$ is a stable model category, i.e.
its homotopy category is triangulated.  Since $\Sigma _T ^\infty Y_+$ is cofibrant
in $Spt (\unsmot _Y)$ and $\pi _{Y\ast}E'$, $f_\ast \pi _{Y' \ast}E'$, $i_\ast \pi _{Z \ast}E'$,
$ j_\ast \pi _{Z' \ast}E'$ are fibrant; 
combining \cite[9.1.6 M7]{Hirsc} and \cite[9.7.5(1)]{Hirsc} we conclude that the induced
commutative diagram is a homotopy fibre square of simplicial sets:
\begin{align*}
  \begin{array}{c}
\xymatrix@C1pc{
Map_Y (\Sigma _T ^\infty Y_+, \pi _{Y\ast}E') \ar[r] \ar[d]&  
Map_Y (\Sigma _T ^\infty Y_+,f_\ast \pi _{Y' \ast}E') \ar[d]\\
		Map_Y (\Sigma _T ^\infty Y_+, i_\ast \pi _{Z \ast}E') \ar[r]& 
		Map_Y (\Sigma _T ^\infty Y_+, j_\ast \pi _{Z' \ast}E')}
  \end{array}
\end{align*}
Since the adjunction $(f^\ast, f_\ast)$ is enriched in simplicial sets, we conclude that:
\[Map_Y (\Sigma _T ^\infty Y_+,f_\ast \pi _{Y' \ast}E')\cong
Map_{Y'}(f^\ast \Sigma _T ^\infty Y_+,\pi _{Y' \ast}E')\cong
Map_{Y'}(\Sigma _T ^\infty Y'_+,\pi _{Y' \ast}E')
\]
and by definition $Map_{Y'}(\Sigma _T ^\infty Y'_+,\pi _{Y' \ast}E')\cong 
Map (\Sigma _T ^\infty Y'_+,E')$.  Similarly, we conclude that 
$Map_Y (\Sigma _T ^\infty Y_+, \pi _{Y\ast}E')\cong Map (\Sigma _T ^\infty Y_+, E')$,
$Map_Y (\Sigma _T ^\infty Y_+, i_\ast \pi _{Z \ast}E')\cong 
Map (\Sigma _T ^\infty Z_+, E')$ and $Map_Y (\Sigma _T ^\infty Y_+, j_\ast \pi _{Z' \ast}E')
\cong Map (\Sigma _T ^\infty Z'_+, E')$.  Therefore, the following is
a homotopy fibre square of simplicial sets:
\begin{align*}
  \begin{array}{c}
\xymatrix@C1pc{
Map(\Sigma _T ^\infty Y_+, E') \ar[r] \ar[d]&  
Map(\Sigma _T ^\infty Y'_+,E') \ar[d]\\
		Map(\Sigma _T ^\infty Z_+, E') \ar[r]& 
		Map(\Sigma _T ^\infty Z'_+, E')}
  \end{array}
\end{align*}
Since $\Sigma _T ^\infty Z'_+ \rightarrow \Sigma _T ^\infty Y'_+$ is a cofibration
in $Spt (\unsmot _{ft})$ and $E'$ is fibrant in $Spt (\unsmot _{ft})$,
we deduce that $Map(\Sigma _T ^\infty Y'_+, E') \rightarrow  Map(\Sigma _T ^\infty Z'_+,E')$
is a fibration of simplicial sets \cite[9.1.6 M7]{Hirsc}.
We observe that the functor $Map(-, E')$ maps pushout squares in $Spt (\unsmot _{ft})$
into pullback squares of simplicial sets \cite[9.1.8]{Hirsc}; thus by
\cite[13.3.8]{Hirsc} we conclude that the map:
\[ Map(\Sigma _T ^\infty Y_+, E')\rightarrow
Map(\Sigma _T ^\infty P(\alpha), E')
\]
induced by $P(\alpha)\rightarrow Y$
is a weak equivalence of simplicial sets, where
$P(\alpha)$ is the pushout of $Z\leftarrow Z'\rightarrow Y'$ in $\underline{Pre}_{k}$.
Finally, by \cite[4.1.1(2)]{Hirsc} we conclude that $E'$ is fibrant in
$\Tspectracdh$ since by construction $\Tspectracdh$ is the left Bousfield localization
of $Spt(\unsmot _{ft})$ with respect to the maps of the form
$\Sigma _T ^{\infty}(P(\alpha)\rightarrow Y_+)$ indexed by the abstract
blow-up squares in $\Sch _k$.
\end{proof}

The following result should be compared with \cite[Prop. 3.7]{Cisinski}.

\begin{thm}\label{thm:basch2}
Let $v: X \to \Spec(k)$ be in $\Sch_k$.
Given a motivic $T$-spectrum $E \in \stablehomotopy$, $Y\in \Sm_X$ 
and integers $m$, $n \in \Z$, there is a natural isomorphism
\[
\Hom_{\stablehomotopy _{X}}(\Sigma^{\infty}_TY_{+},
\Sigma^{m,n} \mathbf L v^{\ast}E)  \cong
\Hom _{\stablehomotopycdh}(\Sigma^{\infty}_TY_{+},\Sigma^{m,n}  
\mathbf L \pi ^{\ast} E).
\]
\end{thm}
\begin{proof}
By ~\ref{prop:basechfor}, 
$\mathbf L v^{\ast}(-) \cong \mathbf (R \pi _{X\ast} \circ \mathbf L \pi
 ^{\ast})(-)$ in $\stablehomotopy _{X}$.  Thus, by adjointness:
\[
\begin{array}{lll}
\Hom_{\stablehomotopy _{X}}(\Sigma^{\infty}_TY_{+},
\Sigma^{m,n} \mathbf L v^{\ast}E)  & \cong &
\Hom_{\stablehomotopy _{X}}(\Sigma^{\infty}_TY_{+},
\mathbf L v^{\ast} (\Sigma^{m,n}E)) \\
& \cong & 
\Hom _{\stablehomotopycdh}(\mathbf L \pi _{X}^{\ast} \Sigma^{\infty}_TY_{+}, 
\mathbf L \pi ^{\ast} (\Sigma^{m,n}E)) \\
& \cong & 
\Hom _{\stablehomotopycdh}(\mathbf L \pi _{X}^{\ast} \Sigma^{\infty}_TY_{+}, 
\Sigma^{m,n} \mathbf L \pi ^{\ast} E).
\end{array}
\]

Finally, it follows from ~\ref{lem:cofibrant-fl} that 
$\Sigma^{\infty}_TY_{+}$ is cofibrant in the levelwise flasque
model structure and hence in any of its localizations.
In particular, it is cofibrant in the stable model structure of motivic
$T$-spectra. We conclude that 
$\mathbf L \pi _{X}^{\ast} \Sigma^{\infty}_TY_{+} \cong
\pi _{X}^{\ast} \Sigma^{\infty}_TY_{+} \cong \Sigma^{\infty}_TY_{+}$.
The corollary now follows. 
\end{proof}

\begin{remk}\label{remk:desc-abs}
The above result could be called a $cdh$-descent theorem because it
implies $cdh$-descent for many motivic spectra
(see \cite[Prop. 3.7]{Cisinski}). In particular, it
implies $cdh$-descent for absolute motivic spectra (e.g., $\KGL$ and $\MGL$). 
Recall from \cite[\S 1.2]{Deglise} that an absolute motivic spectrum $E$ is
a section of a 2-functor from $\Sch_k$ to triangulated categories
such that for any $f: X' \to X$ in $\Sch_k$, the canonical map
$f^*E_X \to E_{X'}$ is an isomorphism.
\end{remk}

\begin{lem}\label{lem:desc-ext-0}
Let $f: Y \to X$ be a smooth morphism in $\Sch_k$. Let
$v: X \to \Spec(k)$ be the structure map and $u = v \circ f$.
Given any $E \in \stablehomotopy$,
the map $\Hom_{\stablehomotopy _{X}}(\Sigma^{\infty}_TY_{+}, \mathbf L v^{\ast}E)
\to \Hom_{\stablehomotopy _{Y}}(\Sigma^{\infty}_TY_{+}, \mathbf L u^{\ast}E)$
is an isomorphism.
\end{lem}
\begin{proof}
The functor
$\mathbf L f^{\ast}:\stablehomotopy_X \rightarrow \stablehomotopy _Y$
admits a left adjoint $\mathbf L f_{\sharp}:\stablehomotopy _Y \rightarrow 
\stablehomotopy_X$ by \cite[Prop. 4.5.19]{Ay-2} (see also \cite[Scholium 1.4.2]{Ay}).
Since $f: Y \to X$ is smooth, we have 
$\mathbf L f_{\sharp}(\Sigma^{\infty}_TY_{+}) =
\Sigma^{\infty}_TY_{+}$ by \cite[Prop.~3.1.23(1)]{MV} and we get
\[
\begin{array}{lll}
\Hom_{\stablehomotopy _{X}}(\Sigma^{\infty}_TY_{+}, \mathbf L v^{\ast}E)
& \cong & \Hom_{\stablehomotopy _{X}}(\mathbf L f_{\sharp}
(\Sigma^{\infty}_TY_{+}), \mathbf L v^{\ast}E) \\
& \cong & 
\Hom_{\stablehomotopy _{Y}}(\Sigma^{\infty}_TY_{+}, \mathbf L f^{\ast} \circ 
\mathbf L v^{\ast}E) \\
& \cong & 
\Hom_{\stablehomotopy _{Y}}(\Sigma^{\infty}_TY_{+}, \mathbf L u^{\ast}E)
\end{array}
\]
and the lemma follows.
\end{proof}

A combination of ~\ref{lem:desc-ext-0} and ~\ref{thm:basch2} yields:

\begin{cor}\label{cor:smbasch2}
Under the same hypotheses and notation of \ref{thm:basch2}, assume in addition that 
$X\in \Sm _k$.  Then 
there are natural isomorphisms:
\[
\Hom_{\stablehomotopy}(\Sigma^{\infty}_TY_{+},
\Sigma^{m,n} E) \cong
\Hom_{\stablehomotopy _{X}}(\Sigma^{\infty}_TY_{+},
\Sigma^{m,n} \mathbf L v^{\ast}E)  \cong
\Hom _{\stablehomotopycdh}(\Sigma^{\infty}_TY_{+},\Sigma^{m,n}  
\mathbf L \pi ^{\ast} E).
\]
\end{cor}

\section{Motivic cohomology of singular schemes}\label{sec:MCS}
We continue to assume that $k$ is a perfect field of exponential 
characteristic $p$. In this section, we show that the motivic cohomology
of a scheme $X \in \Sch_k$, defined in terms of a $cdh$-hypercohomology
(see ~\ref{defn:MC-sing}), is representable in the stable homotopy category 
$\stablehomotopycdh$.

Recall from \cite[Lecture~16]{MVW} that given $T \in \Sch_k$ and an integer
$r \ge 0$, the presheaf $z_{equi}(T,r)$ on $\Sm_k$ is defined by
letting $z_{equi}(T,r)(U)$ be the free abelian group generated by the closed
and irreducible subschemes $Z \subsetneq U \times T$ which are dominant and
equidimensional of relative dimension $r$ (any fiber is either empty or all
its components have dimension $r$) over a component of $U$. 
It is known that $z_{equi}(T,r)$ is a sheaf on the big {\'e}tale site of $\Sm_k$.

Let $\underline{C}_{\ast}z_{equi}(T,r)$ denote the chain complex of
presheaves of abelian groups associated via the Dold-Kan correspondence to
the simplicial presheaf on $\Sm_k$ given by
$\underline{C}_nz_{equi}(T,r)(U) = z_{equi}(T,r)(U \times \Delta^n_k)$.
The simplicial structure on $\underline{C}_{\ast}z_{equi}(T,r)$ 
is induced by the cosimplicial scheme  $\Delta^{\bullet}_k$. 
Recall the following definition of motivic cohomology of singular schemes
from \cite[Def. 9.2]{FV}.

\begin{defn}\label{defn:MC-sing}
The motivic cohomology groups of $X \in \Sch_k$ are defined as the
hypercohomology 
\[
\begin{array}{lll}
H^m(X, \Z(n))  =   \H^{m-2n}_{cdh}(X,  \underline{C}_{\ast}z_{equi}(\A^n_k,0)_{cdh})= A_{0,2n-m}(X,\mathbb A^n).
\end{array}
\]
\end{defn}

We will also need to consider $\mathbb Z [\tfrac{1}{p}]$-coefficients.
In this case, we will write:
\[  H^m(X, \Z [\tfrac{1}{p}](n))=
\H^{m-2n}_{cdh}(X,  \underline{C}_{\ast}z_{equi}(\A^n_k,0)[\tfrac{1}{p}]).
\]

For $n < 0$, we set $H^m(X, \Z(n)) =  H^m(X, \Z [\tfrac{1}{p}](n)) = 0$.

\subsection{The motivic cohomology spectrum}\label{sec:EMSpec}
In order to represent the motivic cohomology of a singular scheme $X$ in 
$\stablehomotopy_X$, let us recall the Eilenberg-MacLane spectrum 
\[
H\Z = (K(0,0), K(1, 2), \cdots , K(n,2n), \cdots )
\]
in $\Tspectra$, where
$K(n,2n)$ is the presheaf of simplicial abelian groups on $\Sm_k$ associated to
the presheaf of chain complexes 
$\underline{C}_{\ast}z_{equi}(\A ^n_k,0)$ via the Dold-Kan
correspondence. The external product of cycles induces product maps
$K(m,2m)\wedge K(n,2n)\rightarrow K(m+n,2(m+n))$.  Notice that
$K(1,2)\cong \underline{C}_{\ast}(z_{equi}(\P^1_k, 0)/z_{equi}(\P^0_k, 0))$
\cite[16.8]{MVW}, so
composing the product maps with the canonical map
$g: T\cong \P^1 _k /\P ^0_k \to 
\underline{C}_{\ast}(z_{equi}(\P^1_k, 0)/z_{equi}(\P^0_k, 0))\cong K(1,2)$
(where the first map assigns to any morphism $U \to \P^1_k$ its graph
in $U \times \P^1_k$),
we obtain the bonding maps. $H\Z$ is a symmetric spectrum
whose symmetric structure is obtained by permuting the coordinates in 
$\A ^n _k$.
We shall not distinguish between a simplicial abelian group and the associated
chain complex of abelian groups from now on in this text and will use them
interchangeably.

\subsection{Motivic cohomology via $\stablehomotopycdh$}
\label{sec:Rep-MCoh}

Let $\mathbf{1}=\Sigma _T ^\infty (S^0_s)$ be the sphere spectrum in $\stablehomotopy$,
and let $\mathbf{1} [\tfrac{1}{p}] \in \stablehomotopy$  be the
homotopy colimit \cite[1.6.4]{Neeman-2} of the filtering diagram in $\stablehomotopy$:
\[  \xymatrix{\mathbf{1} \ar[r]^-{p}& \mathbf{1} \ar[r]^-{p}& \mathbf{1} \ar[r]^-{p}&\cdots}
\]
where $\mathbf{1} \stackrel{r}{\rightarrow} \mathbf{1}$ is the composition
of the sum map with the diagonal: $\mathbf{1} \stackrel{\Delta}{\rightarrow}
\oplus _{i=1}^r \mathbf{1} \stackrel{\Sigma}{\rightarrow} \mathbf{1}$.  
For $E\in \stablehomotopy$, we define
$E [\tfrac{1}{p}]\in \stablehomotopy$ to be $E \wedge \mathbf{1} [\tfrac{1}{p}]$.
This also makes sense in $\stablehomotopy _{X}$ and $\stablehomotopy _{cdh}$.

The following is a reformulation of the main result in \cite{FV} when $k$ admits
resolution of singularities, and the main result in \cite{Kelly} when $k$ has
positive characteristic.

\begin{thm}[\cite{CisDegtr}]\label{thm:cdhdescent}
Let $k$ be a perfect field of exponential characteristic $p$,
and let $v: X \to \Spec(k)$ be a separated scheme of finite type. 
Then for any $m$, $n \in \Z$, there is
a natural isomorphism:
\begin{equation}\label{eqn:cdhdescent-0}
\theta_X: H^m(X, \Z [\tfrac{1}{p}](n)) \xrightarrow{\cong} 
\Hom_{\stablehomotopy_X}(\Sigma^{\infty}_T X_+, 
\Sigma^{m,n} \mathbf L v^* H\Z [\tfrac{1}{p}]).
\end{equation}
\end{thm}
\begin{proof}
Recall that $H^m(X, \Z [\tfrac{1}{p}](n))=A_{0,2n-m}(X,\mathbb A^n)$ \eqref{defn:MC-sing}.
We observe that $\underline{C}_{\ast}z_{equi}(\A^n_k,0)$ is the motive with compact supports $M^c(\A ^n _k)$
of $\A ^n _k$ \cite[\S 4.1]{Voev-trmot}, \cite[16.13]{MVW}.  Combining \cite[Cor. 4.1.8]{Voev-trmot} (or
\cite[16.7, 16.14]{MVW})
with \cite[4.2, Prop. 4.3, Thm. 5.1 and Cor. 8.6]{CisDegtr}, we conclude that
\[H^m(X, \Z [\tfrac{1}{p}](n))\cong \Hom_{\stablehomotopy_X}(\Sigma ^{2n-m,0}(\Sigma^{\infty}_T X_+),  
\Sigma^{2n,n} \mathbf L v^* H\Z [\frac{1}{p}])
\]
which finishes the proof.
\end{proof}

As a combination of ~\ref{thm:basch2} and ~\ref{thm:cdhdescent},
we get

\begin{cor}\label{thm:cdhdescent-2}
Under the hypothesis and with the notation of \ref{thm:cdhdescent},
there are natural isomorphisms
\[
\begin{array}{lll}
H^m(X, \Z [\frac{1}{p}](n)) & \cong &
\Hom_{\stablehomotopycdh}(\Sigma^{\infty}_T X_+, \Sigma^{m,n} 
\mathbf L \pi^* H\Z [\frac{1}{p}])  \\
& \cong & \Hom_{\stablehomotopy_X}(\Sigma^{\infty}_T X_+,  
\Sigma^{m,n} \mathbf L v^* H\Z [\frac{1}{p}]).
\end{array}
\]
\end{cor}

\section{Slice spectral sequence for singular schemes}
\label{sec:SSS}
Let $k$ be a perfect field of exponential characteristic $p$.
Given $X \in \Sch_k$, recall that Voevodsky's slice filtration of
$\stablehomotopy _{X}$ is given as follows.
For an integer $q \in \Z$, let $\Sigma^{q}_{T}\stablehomotopy _{X}^{\fnteff}$ 
denote the smallest full triangulated subcategory of $\stablehomotopy _{X}$ 
which contains $C^{q}_{\fnteff}$ and is closed under arbitrary coproducts, where
\begin{equation}\label{eq:effcat}
C^{q}_{\fnteff}=\{ F_{n}(S^{r}_s\wedge S^{s}_t \wedge Y_{+}): n, r, s\geq 0, 
s-n\geq q, Y\in \Sm_X\}.  
\end{equation}

In particular, $\stablehomotopy _{X}^{\fnteff}$ is the smallest 
full triangulated subcategory of $\stablehomotopy _{X}$ which is closed
under infinite direct sums and contains all spectra of the type
$\Sigma^{\infty}_TY_+$ with $Y \in \Sm_X$.
The slice filtration of $\stablehomotopy _{X}$ (see \cite{Voev-2}) is the
sequence of full triangulated subcategories
\[ 
\cdots \subseteq \susp{q+1}_{T}\stablehomotopy _{X}^{\fnteff} \subseteq
\susp{q}_{T}\stablehomotopy _{X}^{\fnteff} \subseteq \susp{q-1}_{T}
\stablehomotopy _{X}^{\fnteff} \subseteq \cdots
\]

It follows from the work of Neeman \cite{Neeman-1}, \cite{Neeman-2} that the 
inclusion $i_{q}:\susp{q}_{T}\stablehomotopy _{X}^{\fnteff}\rightarrow 
\stablehomotopy_X$ admits a right adjoint 
$r_{q}:\stablehomotopy_X \rightarrow \susp{q}_{T}\stablehomotopy _{X}^{\fnteff}$
and that the functors 
$f_{q}, s_{<q}, s_{q}:\stablehomotopy_X \rightarrow \stablehomotopy_X$
are triangulated; where $r_q \circ i_q$ is the identity,
$f_{q}=i_{q}\circ r_{q}$ and $s_{<q}$, $s_{q}$ are 
characterized by the existence of the following distinguished triangles in 
$\stablehomotopy _{X}$:
\begin{equation}\label{eqn:Slice-triangle}
\xymatrix@R=0.8pc{f_{q}E\ar[r] & E \ar[r]& s_{<q}E \\
			f_{q+1}E\ar[r]& f_{q}E \ar[r]& s_{q}E}
\end{equation}
for every $E\in \stablehomotopy _{X}$.

\begin{defn}  \label{def.indfil}
Let $a$, $b$, $n \in \mathbb Z$ and $Y\in Sm_X$.
Let $F^{n}E^{a,b}(Y)$ be the image of the map induced by $f_{n}E\rightarrow E$
\eqref{eqn:Slice-triangle}:
$\Hom _{\stablehomotopy _X}(\Sigma _T ^\infty Y_+, \Sigma ^{a,b}f_n E)\rightarrow
\Hom _{\stablehomotopy _X}(\Sigma _T ^\infty Y_+, \Sigma ^{a,b}E)$.  This 
determines a
decreasing filtration $F^{\bullet}$ on $E^{a,b}(Y)=\Hom _{\stablehomotopy _X}
(\Sigma _T ^\infty Y_+, \Sigma ^{a,b}E)$, and we will write $gr^{n}F^\bullet$ for 
the associated graded
$F^{n}E^{a,b}(Y)/F^{n+1}E^{a,b}(Y)$.
\end{defn}

The following result is well known \cite[\S 2]{Voev-2}.

\begin{prop}  \label{pr.filexhs}
The filtration $F^\bullet$ on $E^{a,b}(Y)$ is exhaustive 
in the sense of \cite[Def. 2.1]{Boardman}.
\end{prop}
\begin{proof}
Recall that $\stablehomotopy _X$ is a compactly generated triangulated category 
in the sense of Neeman  \cite[Def. 1.7]{Neeman-1}, with set of compact generators 
\cite[Thm. 4.5.67]{Ay-2}: $\cup _{q\in \mathbb Z} C^q _{\fnteff}$
\eqref{eq:effcat}.  Therefore a map $f:E_1\rightarrow E_2$ in
$\stablehomotopy _X$ is an isomorphism if and only if for every $Y\in \Sm _X$
and every $m$, $n\in \mathbb Z$ the induced map of abelian groups
$\Hom_{\stablehomotopy _X}(\Sigma ^{m,n}Y_+, E_1) \rightarrow
\Hom_{\stablehomotopy _X}(\Sigma ^{m,n}Y_+, E_2)$
is an isomorphism.  Thus, we conclude that
$E\cong \mathrm{hocolim} f_qE$  in $\stablehomotopy _X$.

Therefore, we deduce that for every
$a$, $b\in \mathbb Z$ and every $Y\in \Sm _X$; 
there exist the following isomorphisms \cite[Lem. 2.8]{Neeman-1},
\cite[Thm. 6.8]{Isak}:
\begin{align*} 
\colim _{n\rightarrow -\infty} F^n E^{a,b}(Y) & \cong \colim _{n\rightarrow -\infty} 
\Hom _{\stablehomotopy _X} (\Sigma ^{\infty} _T Y_+, \Sigma ^{a,b}f_nE)\\ & \cong 
\Hom _{\stablehomotopy _X} (\Sigma ^{\infty} _T Y_+, \Sigma ^{a,b} 
\mathrm{hocolim}f_qE) \cong E^{a,b}(Y),
\nonumber
\end{align*}
so the filtration $F^\bullet$ is exhaustive.
\end{proof}

\subsection{The slice spectral sequence}\label{sec:SS-Conv}
Let $Y\in \Sm_X$ be a smooth $X$-scheme and $G \in \stablehomotopy _{X}$.  
Since $\stablehomotopy _{X}$ is a triangulated category, the collection of 
distinguished triangles
$\{ f_{q+1}G\rightarrow f_{q}G\rightarrow s_{q}G\}_{q\in \mathbb Z}$ determines a 
(slice) spectral sequence: 
\[E_1^{p,q}=\Hom _{\stablehomotopy _{X}}(\Sigma _T ^{\infty}
Y_+,\Sigma _s ^{p+q}s_p G)
\]
with $G^{\ast, \ast}(Y)$ as its abutment and differentials 
$d_r: E_r^{p,q}\rightarrow E_r^{p+r, q-r+1}$.

In order to study the convergence of this spectral sequence,
recall from \cite[p.~22]{Voev-2} that $G\in \stablehomotopy _{X}$ is called 
{\sl bounded} with respect to the slice filtration if for every 
$m$, $n\in \mathbb Z$ and  every $Y\in \Sm_X$, there exists 
$q\in \mathbb Z$ such that
\begin{equation}  \label{eqn:bound}
\Hom _{\stablehomotopy _{X}}(\susp{m,n}\Sigma^{\infty}_TY_{+}, f_{q+i}G)=0
\end{equation}
for every $i>0$.  Clearly the slice spectral sequence 
is strongly convergent when $G$ is bounded.

\begin{prop}\label{prop:bound-0}
Let $k$ be a field with resolution of singularities.
Let $F \in \stablehomotopy$ be bounded with respect
to the slice filtration and let 
$G=\mathbf L v^{\ast}F\in \stablehomotopy _{X}$, where
$v:X\rightarrow \Spec k$.  Then
$G$ is bounded with respect to the slice filtration.
\end{prop}
\begin{proof}
Since the base field $k$ admits resolution of singularities, we
deduce by \cite[Thm.~3.7]{Pelaez-3} that 
$f_{q}G\cong \mathbf L v ^{\ast}f_{q}F$ in $\stablehomotopy _{X}$ for every
$q\in \mathbb Z$.  
It follows from ~\ref{thm:basch2} that for every 
$m$, $n\in \mathbb Z$; and  every $Y\in \Sm_X$, we have
\[  
\Hom _{\stablehomotopy _{X}}(\susp{m,n}\Sigma^{\infty}_TY_{+}, f_{q+i}G)\cong 
\Hom _{\stablehomotopycdh}(\susp{m,n}\Sigma^{\infty}_TY_{+}, 
\mathbf L \pi ^{\ast}(f_{q+i}F))
\]
for every $i>0$.  
If $X\in \Sm_k$, then $Y\in \Sm_k$ and we have
\[  \Hom _{\stablehomotopycdh}(\susp{m,n}\Sigma^{\infty}_TY_{+}, 
\mathbf L \pi ^{\ast}(f_{q+i}F))  \cong
\Hom _{\stablehomotopy}(\susp{m,n}\Sigma^{\infty}_TY_{+}, f_{q+i}F)
\]
for every $i>0$ by  \ref{cor:smbasch2}.
Since $F$ is bounded with respect to the slice filtration, 
we deduce from ~\eqref{eqn:bound} that $G$ is also bounded in 
$\stablehomotopy _{X}$ in this case.

Finally, we proceed by induction on the dimension of $Y$, and assume that for 
every
$m$, $n\in \mathbb Z$ and every $Y'\in \Sch_k$ with $\dim(Y') < \dim(Y)$; 
there exists
$q\in \mathbb Z$ such that
\[  \Hom _{\stablehomotopycdh}(\susp{m,n}\Sigma^{\infty}_TY'_{+}, 
\mathbf L \pi ^{\ast}(f_{q+i}F))=0
\]
for every $i>0$.  Since the base field $k$ admits resolution of singularities, 
there exists a $cdh$-cover 
$\{X' \amalg Z \rightarrow Y\}$ of $Y$ such that
$X' \in \Sm_k$, $\dim(Z) < \dim(Y)$ and $\dim(W) < \dim(Y)$, where  
we set $W = X' \times _{Y} Z$.

Let $q_{1}$ (resp. $q_{2}$, $q_{3}$) be the integers such
that the vanishing condition \eqref{eqn:bound}  holds for $(X',m,n)$ 
(resp. $(Z,m,n)$, $(W,m+1,n)$).
Let $q$ be the maximum of $q_{1}$, $q_{2}$ and $q_{3}$.  Then by $cdh$-excision,
for every $i>0$, the following diagram is exact:
\begin{align*}
\Hom _{\stablehomotopycdh}(\susp{m+1,n}\Sigma^{\infty}_TW, 
\mathbf L \pi ^{\ast}(f_{q+i}F))\rightarrow 
\Hom _{\stablehomotopycdh}(\susp{m,n}\Sigma^{\infty}_TY_{+}, 
\mathbf L \pi ^{\ast}(f_{q+i}F)) \rightarrow \\
\Hom _{\stablehomotopycdh}(\susp{m,n}\Sigma^{\infty}_TX'_{+}, 
\mathbf L \pi ^{\ast}(f_{q+i}F))\oplus 
\Hom _{\stablehomotopycdh}(\susp{m,n}\Sigma^{\infty}_TZ_{+}, 
\mathbf L \pi ^{\ast}(f_{q+i}F)).
\end{align*}

By the choice of $q$, both ends in the diagram vanish, hence the group in the 
middle also vanishes as we wanted.
\end{proof}

In order to get convergence results in positive characteristic, we need to restrict to
spectra $E\in \stablehomotopy$ which admit a structure of traces 
\cite[4.2.27 and 4.3.1]{Kelly}.

\begin{lem}  \label{lem.compcoef1}
With the notation of \ref{eqn:commdiag2},
let $X\in \Sch _k$.  Then:
\begin{enumerate}
\item \label{a.coef1} For every $E\in \stablehomotopy$, 
$\mathbf L \pi ^{\ast}(E[\tfrac{1}{p}])\cong 
(\mathbf L \pi ^{\ast}E)[\tfrac{1}{p}]$ and $\mathbf L v ^{\ast}(E[\tfrac{1}{p}])\cong 
(\mathbf L v ^{\ast}E)[\tfrac{1}{p}]$.
\item \label{b.coef1} For every $E\in \stablehomotopy _{cdh}$ and every $a$, $b\in \mathbb Z$:
\begin{align*} 
\Hom_{\stablehomotopy_{cdh}}(\Sigma ^{a,b}\Sigma _T^\infty (X_+),E[\tfrac{1}{p}])
& \cong \Hom_{\stablehomotopy_{cdh}}(\Sigma ^{a,b}\Sigma _T^\infty (X_+),E)\otimes 
\mathbb Z [\tfrac{1}{p}].
\end{align*}
\end{enumerate}
\end{lem}
\begin{proof}
\eqref{a.coef1}:  It follows from the definition of homotopy colimit \cite[1.6.4]{Neeman-2} that
$\mathbf L \pi ^{\ast}$ and $\mathbf L v^{\ast}$ commute with homotopy colimits since
they are left adjoint.  This implies the result since $E[\tfrac{1}{p}]$ is
given in terms of homotopy colimits.

\eqref{b.coef1}:  Since $\Sigma ^{a,b}\Sigma _T ^{\infty}(X_+)$ is compact in $\stablehomotopy _{cdh}$
\cite[Thm. 4.5.67]{Ay-2}, the result follows from \cite[Lem. 2.8]{Neeman-1}.
\end{proof}

\begin{lem}  \label{lem.compcoef2}
Let $X \in \Sch _k$ and $E\in \stablehomotopy _X$.
Then for every $r\in \mathbb Z$, $f_r(E[\tfrac{1}{p}])\cong (f_r E)[\tfrac{1}{p}]$
and $s_r(E[\tfrac{1}{p}])\cong (s_r E)[\tfrac{1}{p}]$.
\end{lem}
\begin{proof}
Since the effective categories $\Sigma^{q}_{T}\stablehomotopy _{X}^{\fnteff}$
are closed under infinite direct sums, we conclude that the functors $f_r$, $s_r$
commute with homotopy colimits.
\end{proof}

\begin{prop}\label{prop:bound-charp}
Let $F \in \stablehomotopy$ and $G=\mathbf L v^{\ast}F\in \stablehomotopy _{X}$,
where $v:X\rightarrow \Spec k$. 
Assume that for every $r\in \mathbb Z$, $s_r(F[\frac{1}{p}])$
has a weak structure of smooth traces
in the sense of \cite[4.2.27]{Kelly}; and that $F[\frac{1}{p}]$
has a structure of traces in the
sense of \cite[4.3.1]{Kelly}.  If $F[\frac{1}{p}]$
is bounded with respect to the slice filtration, then  
$G[\frac{1}{p}]$ is bounded as well.
\end{prop}
\begin{proof}
Since the base field $k$ is perfect and $F[\frac{1}{p}]$ is clearly 
$\mathbb Z [\tfrac{1}{p}]$-local, 
combining \cite[4.2.29]{Kelly} and \ref{lem.compcoef2}, we conclude that 
$f_{q}G[\frac{1}{p}]\cong \mathbf L v ^{\ast}f_{q}F[\frac{1}{p}]$ in $\stablehomotopy _{X}$ 
for every
$q\in \mathbb Z$.

It follows from \ref{thm:basch2} that for every 
$m$, $n\in \mathbb Z$; and  every $Y\in \Sm_X$, we have
\[  
\Hom _{\stablehomotopy _{X}}(\susp{m,n}\Sigma^{\infty}_T(Y_{+}), 
f_{q+i}G[\tfrac{1}{p}])\cong 
\Hom _{\stablehomotopycdh}(\susp{m,n}\Sigma^{\infty}_T(Y_{+}), 
\mathbf L \pi ^{\ast}(f_{q+i}F[\tfrac{1}{p}])
\]
for every $i>0$.  
If $X\in \Sm_k$, then $Y\in \Sm_k$ and we have
\[  \Hom _{\stablehomotopycdh}(\susp{m,n}\Sigma^{\infty}_T(Y_{+}), 
\mathbf L \pi ^{\ast}(f_{q+i}F[\tfrac{1}{p}]))  \cong
\Hom _{\stablehomotopy}(\susp{m,n}\Sigma^{\infty}_T(Y_{+}), f_{q+i}F[\tfrac{1}{p}])
\]
for every $i>0$ by  \ref{cor:smbasch2}.
Since $F[\frac{1}{p}]$ is bounded with respect to the slice filtration, 
we deduce from ~\eqref{eqn:bound} that $G[\frac{1}{p}]$ is also bounded 
with respect to the slice filtration in $\stablehomotopy _{X}$ in this case.

Finally, we proceed by induction on the dimension of $Y$, and assume that for 
every
$m$, $n\in \mathbb Z$ and every $Z\in \Sch_k$ with $\dim_k(Z) < \dim_k(Y)$; 
there exists
$q\in \mathbb Z$ such that
\[  \Hom _{\stablehomotopycdh}(\susp{m,n}\Sigma^{\infty}_T(Z_{+}), 
\mathbf L \pi ^{\ast}(f_{q+i}F[\tfrac{1}{p}]))=0
\]
for every $i>0$.

Since $k$ is perfect,
by a theorem of Gabber \cite[Introduction Thm. 3(1)]{Gabber} and
Temkin's strengthening \cite[Thm. 1.2.9]{Temkin} of Gabber's result,
there exists $W\in \Sm _k$
and a surjective proper map $h:W\rightarrow Y$, which is generically \'etale of degree
$p^r$, $r\geq 1$.  In particular $h$ is generically flat, thus by a theorem of Raynaud-Gruson
\cite[Thm. 5.2.2]{Ray}, there exists a blow-up $g:Y'\rightarrow Y$ with center $Z$
such that the following diagram commutes, where $h'$ is finite flat surjective 
of degree $p^r$
and $g':W'\rightarrow W$ is the blow-up of W with center $h^{-1}(Z)$:
\begin{align}  \label{bwupsqr}
\begin{array}{c}
\xymatrix{W' \ar[d]_-{g'} \ar[r]^-{h'}& Y' \ar[d]^-{g} \\
	W \ar[r]_-{h} & Y.}
\end{array}
\end{align}

Thus we have a $cdh$-cover 
$\{Y' \amalg Z \rightarrow Y\}$ of $Y$, such that $\dim_k(Z) < \dim_k(Y)$ and 
$\dim_k(E) < \dim_k(Y)$, where we set $E = Y' \times _{Y} Z$.

Let $q_{1}$ (resp. $q_{2}$, $q_{3}$) be the integers such
that the vanishing condition \eqref{eqn:bound}  holds for $(W,m,n)$ 
(resp. $(Z,m,n)$, $(E,m+1,n)$).
Let $q$ be the maximum of $q_{1}$, $q_{2}$ and $q_{3}$.  Then by $cdh$-excision,
for every $i>0$, the following diagram is exact:
\begin{align*}
\Hom _{\stablehomotopycdh}(\susp{m+1,n}\Sigma^{\infty}_T(E _+), 
\mathbf L \pi ^{\ast}(f_{q+i}F[\tfrac{1}{p}]))\rightarrow 
\Hom _{\stablehomotopycdh}(\susp{m,n}\Sigma^{\infty}_T(Y_{+}), 
\mathbf L \pi ^{\ast}(f_{q+i}F[\tfrac{1}{p}])) \rightarrow \\
\Hom _{\stablehomotopycdh}(\susp{m,n}\Sigma^{\infty}_T(Y'_{+}), 
\mathbf L \pi ^{\ast}(f_{q+i}F[\tfrac{1}{p}]))\oplus 
\Hom _{\stablehomotopycdh}(\susp{m,n}\Sigma^{\infty}_T(Z_{+}), 
\mathbf L \pi ^{\ast}(f_{q+i}F[\tfrac{1}{p}])).
\end{align*}

By the choice of $q$, this  reduces to the following exact diagram:
\[ \xymatrix{0 \ar[r]& \Hom _{\stablehomotopycdh}(\susp{m,n}\Sigma^{\infty}_T(Y_{+}), 
\mathbf L \pi ^{\ast}(f_{q+i}F[\tfrac{1}{p}])) \ar[r]^-{g^\ast} &  
\Hom _{\stablehomotopycdh}(\susp{m,n}\Sigma^{\infty}_T(Y'_{+}), 
\mathbf L \pi ^{\ast}(f_{q+i}F[\tfrac{1}{p}])).}
\]
So it suffices to show that $g^\ast =0$.  In order to prove this, we observe
that the diagram \ref{bwupsqr} commutes, therefore by the choice of $q$:
$\Hom _{\stablehomotopycdh}(\susp{m,n}\Sigma^{\infty}_T(W_{+}), 
\mathbf L \pi ^{\ast}(f_{q+i}F[\frac{1}{p}]))=0$; and we conclude that 
$h^{\prime \ast} \circ g^\ast = g^{\prime \ast }\circ h^\ast=0$.
Thus, it is enough to see that
\[ h^{\prime \ast}: \Hom _{\stablehomotopycdh}(\susp{m,n}\Sigma^{\infty}_T(Y'_{+}), 
\mathbf L \pi ^{\ast}(f_{q+i}F[\tfrac{1}{p}]))\rightarrow
\Hom _{\stablehomotopycdh}(\susp{m,n}\Sigma^{\infty}_T(W'_{+}), 
\mathbf L \pi ^{\ast}(f_{q+i}F[\tfrac{1}{p}]))
\]
is injective.  Let $v':Y'\rightarrow \Spec k$, and $\epsilon: \mathbf L v^{\prime \ast}(f_{q+i}F[\frac{1}{p}])
\rightarrow \mathbf R h'_{\ast} \mathbf L h^{\prime \ast} \mathbf L v^{\prime \ast}(f_{q+i}F[\frac{1}{p}])$
be the map given by the unit of the adjunction $(\mathbf L h^{\prime \ast}, \mathbf R h'_{\ast})$.
By the naturality of the isomorphism in 2.13 we deduce that $h^{\prime \ast}$ gets identified
with the map induced by $\epsilon$:
\[ \epsilon _{\ast}: \Hom _{\stablehomotopy _{Y'}}(\susp{m,n}\Sigma^{\infty}_T Y'_{+}, 
\mathbf L v^{\prime \ast}f_{q+i}F[\tfrac{1}{p}]) \rightarrow
\Hom _{\stablehomotopy _{Y'}}(\susp{m,n}\Sigma^{\infty}_T Y'_{+}, 
\mathbf R h'_{\ast} \mathbf L h^{\prime \ast} \mathbf L v^{\prime \ast}f_{q+i}F[\tfrac{1}{p}])
\]
Since $F[\frac{1}{p}]$ has a structure of traces and $s_{r}(F[\frac{1}{p}])$ has a weak
structure of smooth traces for every $r\in \mathbb Z$, it follows from 
\cite[4.3.7]{Kelly} that $f_{q+i}(F[\frac{1}{p}])$ has a structure of traces in the sense of
\cite[4.3.1]{Kelly}.  Thus, we deduce from \cite[4.3.1(Deg) p. 101]{Kelly} that $\epsilon _{\ast}$
is injective since $h'$ is finite flat surjective of degree $p^r$.
This finishes the proof.
\end{proof}

If we only assume that the slices $s_r E$ have a structure of traces, then we get
the weaker conditions \ref{pr:findiff.cp}.

\begin{cor}\label{cor:boundsli}
Let $F \in \stablehomotopy$ and $G=\mathbf L v^{\ast}F\in \stablehomotopy _{X}$,
where $v:X\rightarrow \Spec k$ is the structure map. 
Assume that the following hold.
\begin{enumerate}
\item
For every $r\in \mathbb Z$, $s_r(F[\frac{1}{p}])$
has a structure of traces in the sense of \cite[4.3.1]{Kelly}.  
\item
$F[\frac{1}{p}]$
is bounded with respect to the slice filtration.
\end{enumerate}
Then for every $m$, $n$
in $\mathbb Z$, there exists $q\in \mathbb Z$ such that 
$\Hom_{\stablehomotopy _{X}}(\Sigma ^{m,n}
\Sigma _T ^{\infty}X_+,s_{q+i} G[\frac{1}{p}])=0$ for every $i>0$ (see \ref{eqn:bound}).
\end{cor}
\begin{proof}
Since $s_r(F[\tfrac{1}{p}])$ has a structure of traces \cite[4.3.1]{Kelly},  we
observe that in particular $s_r(F[\tfrac{1}{p}])$ has a weak
structure of smooth traces \cite[4.2.27]{Kelly}.  Thus combining
\ref{lem.compcoef1}, \ref{lem.compcoef2} and \cite[4.2.29]{Kelly}
we conclude that for every $r\in \mathbb Z$: $s_r G[\tfrac{1}{p}]\cong \mathbf L v^{\ast}
s_r F[\tfrac{1}{p}]$ and $f_r G[\tfrac{1}{p}]\cong \mathbf L v^{\ast}
f_r F[\tfrac{1}{p}]$.

If $X\in \Sm_k$, we have
\[  \Hom _{\stablehomotopy_{X}}(\susp{m,n}\Sigma^{\infty}_T X_{+}, 
\mathbf L v ^{\ast}(s_{q+i}F [\tfrac{1}{p}]))  \cong
\Hom _{\stablehomotopy}(\susp{m,n}\Sigma^{\infty}_T X_{+}, s_{q+i}F [\tfrac{1}{p}])
\]
for every $i>0$ by  \ref{cor:smbasch2}.  Since $F[\tfrac{1}{p}]$ is bounded with respect
to the slice filtration, there exist $q_1$, $q_2\in \mathbb Z$ such that the vanishing
condition \eqref{eqn:bound} holds for $(X,m,n)$, $(X,m-1,n)$ respectively.
Let $q$ be the maximum of $q_1$ and $q_2$, then using the distinguished triangle
$f_{q+i}F [\tfrac{1}{p}]\rightarrow s_{q+i}F[\tfrac{1}{p}]\rightarrow \Sigma _s ^1 f_{q+i+1}
F[\tfrac{1}{p}]$ in $\stablehomotopy$ we conclude that 
$\Hom _{\stablehomotopy}(\susp{m,n}\Sigma^{\infty}_T(X_{+}), s_{q+i}F[\tfrac{1}{p}])=0$
for every $i>0$ as we wanted.

When $X\in \Sch _k$, the argument in the proof of \ref{prop:bound-charp} works
mutatis-mutandis replacing $f_{q+i} F[\tfrac{1}{p}]$ with $s_{q+i} F[\tfrac{1}{p}]$; since for
every $j\in \mathbb Z$, $s_j F[\tfrac{1}{p}]$ has a structure of traces \cite[4.3.1]{Kelly}.
\end{proof}

\begin{cor}\label{cor:findiff.cp}
Assume that the conditions (1) and (2) of \ref{cor:boundsli} hold.
Then for every $m$, $n$
in $\mathbb Z$, there exists $q\in \mathbb Z$ such that the map
$f_{q+i+1}G[\frac{1}{p}]\rightarrow f_{q+i}G[\frac{1}{p}]$ induces an isomorphism
\begin{align*}
\Hom_{\stablehomotopy _{X}}(\Sigma ^{m,n}\Sigma _T ^{\infty}X_+,
f_{q+i+1} G[\tfrac{1}{p}]) & \cong \Hom_{\stablehomotopy _{X}}
(\Sigma ^{m,n}\Sigma _T ^{\infty}X_+,f_{q+i} G[\tfrac{1}{p}])
\end{align*}
for every $i>0$.
\end{cor}
\begin{proof}
Let $q_1$, $q_2\in \mathbb Z$ be the integers corresponding to $(m,n)$, $(m+1,n)$
in \ref{cor:boundsli}, respectively.  Let $q$ be the maximum of $q_1$ and $q_2$.
Then the result follows by combining the vanishing in \ref{cor:boundsli} with the
distinguished triangle $\Sigma _s ^{-1}s_{q+i}[\tfrac{1}{p}] \rightarrow f_{q+i+1}G[\tfrac{1}{p}]
\rightarrow f_{q+i}G[\tfrac{1}{p}] \rightarrow s_{q+i}G[\tfrac{1}{p}]$ in $\stablehomotopy _X$.
\end{proof}

\begin{remk}  \label{rm.filcomp}
Combining \ref{def.indfil} and \ref{cor:findiff.cp}, we deduce that
for every $a$, $b\in \mathbb Z$, there exists $m\in \mathbb Z$
such that $F^nG[\tfrac{1}{p}]^{a,b}(X)=F^mG[\tfrac{1}{p}]^{a,b}(X)$ for every $n\geq m$.
\end{remk}

\begin{prop}\label{pr:findiff.cp}
Assume that the conditions (1) and (2) of \ref{cor:boundsli} hold.
Then for every $n\in \mathbb Z$,
the slice spectral sequence \eqref{sec:SS-Conv}:
\begin{align*} 
E_1^{a,b}(X,n) & =\Hom_{\stablehomotopy _{X}}(\Sigma _T ^{\infty}X_+,
\Sigma ^{a+b+n,n} s_a G[\tfrac{1}{p}]) \Rightarrow G[\tfrac{1}{p}]^{a+b+n,n}(X)
\end{align*}
satisfies the following.
\begin{enumerate}
\item \label{pr:findiff.cp-1} For every $a$, $b\in \mathbb Z$,
there exists $N>0$ such that $E^{a,b}_r=E^{a,b}_\infty$ for $r\geq N$,
where $E^{a,b}_\infty$
is the associated graded $gr^{a}F^{\bullet}$ with respect
to the descending filtration $F^{\bullet}$ on  
$G[\frac{1}{p}]^{a+b+n,n}(X)$, see \ref{def.indfil}.
\item \label{pr:findiff.cp-2}  For every $m$, $n\in \mathbb Z$,
the descending  filtration $F^{\bullet}$ on 
$G[\frac{1}{p}]^{m,n}(X)$ \eqref{def.indfil}
is exhaustive and complete \cite[Def. 2.1]{Boardman}.
\end{enumerate}
\end{prop}
\begin{proof}
\eqref{pr:findiff.cp-1}: It suffices to show that for every 
$a$, $b\in \mathbb Z$ only
finitely many of the differentials $d_r:E_r^{a,b}\rightarrow E_r^{a+r,b-r+1}$ are nonzero.
But this follows from \ref{cor:boundsli}.

\eqref{pr:findiff.cp-2}: By \ref{pr.filexhs} the filtration $F^\bullet$ on 
$G[\tfrac{1}{p}]^{m,n}(X)$ is exhaustive.
Finally, the completeness of $F^\bullet$ follows by combining \ref{rm.filcomp}
with \cite[Prop. 1.8 and Prop. 2.2(c)]{Boardman}.
\end{proof}

\subsection{The slice spectral sequence for $\MGL(X)$}\label{sec:slicess}
Our aim here is to apply the results of the previous sections to
obtain a Hopkins-Morel type spectral sequence for $\MGL^{\ast, \ast}(X)$
when $X$ is a singular scheme.  For smooth schemes, the Hopkins-Morel
spectral sequence has been studied in \cite{Levine-c1}, \cite{Hoyois},  and over
Dedekind domains in \cite{Sp-2}.

Recall from \cite[\S~6.3]{Voev-0} that for any noetherian scheme $S$ of
finite Krull dimension,
the scheme ${\rm Gr}_S(N, n)$ parametrizes $n$-dimensional linear
subspaces of $\A^N_S$ and one writes ${\rm BGL}_{S,n} = 
\colim_N {\rm Gr}_S(N, n)$. There is a universal rank $n$ bundle $U_{S,n} \to
{\rm BGL}_{S,n}$ and one denotes the Thom space $Th(U_{S,n})$ of this bundle by
$\MGL_{S,n}$. Using the fact that the Thom space of
a direct sum is the smash product of the corresponding Thom spaces
and $T = Th(\sO_S)$, one gets
a $T$-spectrum $\MGL_S = (\MGL_{S,0}, \MGL_{S,1}, \cdots ) \in 
Spt({\sM_S})$.  There is a structure of symmetric spectrum on $\MGL_S$ for 
which we refer to \cite[\S 2.1]{PPR}.

We now let $k$ be a field of characteristic zero
and let $X \in \Sch_k$. We shall use $\MGL$ as a short hand for
$\MGL_k$ throughout this text.
It follows from the above definition of $\MGL_X$ 
(which shows that $\MGL_X$ is
constructed from presheaves represented by smooth schemes)
and \ref{prop:tablebasech}  that the canonical map
$\mathbf L v^*(\MGL) \to \MGL_X$ is an isomorphism.

\begin{defn}\label{defn:pull-back}
We define $\MGL^{\ast, \ast}(X)$ to be the
generalized cohomology groups 
\[
\begin{array}{lll}
\MGL^{p,q}(X) & := & \Hom_{\stablehomotopy_X}(\Sigma^{\infty}_TX_+, 
\Sigma^{p,q} \MGL_X) \\
& \cong & \Hom_{\stablehomotopy_X}(\Sigma^{\infty}_TX_+, 
\Sigma^{p,q} \mathbf L v^*\MGL).
\end{array}
\]
\end{defn}

It follows from ~\ref{thm:basch2} that:
\begin{equation}\label{eqn:pull-back-0}
\MGL^{p, q} (X) \cong 
\Hom_{\stablehomotopycdh}(\Sigma^{\infty}_TX_+, 
\Sigma^{p,q} \mathbf L \pi^*\MGL).
\end{equation}

We now construct the spectral sequence for $\MGL^{*,*}(X)$ 
using the exact couple
technique as follows. 
For $p, q, n \in \mathbb{Z}$, define
\[
A^{p,q} (X, n):= [\Sigma^{\infty}_TX_+, \Sigma_{s} ^{p+q-n} \Sigma_{t} ^n 
(f_p \MGL_X)].
\] 
\[
E^{p,q} (X, n):= [\Sigma^{\infty}_TX_+, 
\Sigma_{s} ^{p+q-n} \Sigma_{t} ^n s_p \MGL_X].
\]
Here, $[-,-]$ denotes the morphisms in $\stablehomotopy_X$.
It follows from ~\eqref{eqn:Slice-triangle} that there is an exact
sequence
\begin{equation}\label{eqn:SS-0}
A^{p+1,q-1} (X, n) \xrightarrow{a^{p,q}_n} A^{p,q}(X,n) 
\xrightarrow{b^{p,q}_n} E^{p,q}(X,n) \xrightarrow{c^{p,q}_n}
A^{p+1,q}(X,n).
\end{equation}
 
Set $D_1(X, n):= \oplus_{p,q} A^{p,q} (X, n)$ and  
$E_1 (X, n):= \oplus_{p,q} E^{p,q} (X, n)$. Write 
$a^1_n:= \oplus a^{p,q}_n$, $b^1_n:= \oplus b^{p,q}_n$ and $c^1_n:= 
\oplus c^{p,q}_n$. This gives an exact couple
$\{D^1_n, E^1_n, a^1_n, b^1_n, c^1_n\}$ and
the map $d^1_n = b^1_n \circ c^1_n: E^1_n \to E^1_n$ shows that $(E_1, d_1)$ is 
a complex. Thus, by repeatedly taking the homology functors, we obtain a 
spectral sequence.

For the target of the spectral sequence, 
let $A^m(X, n):= \colim_{q \to \infty} A^{m-q, q} (X, n)$. Since $X$ is a compact 
object of $\stablehomotopy _{X}$ (see \cite[Prop.~5.5]{Voev-0}, 
\cite[Thm.~4.5.67]{Ay-2}), 
the colimit enters into $[-,-]$ so that 
$A^m (X, n)= [\Sigma^{\infty}_TX_+, \Sigma_{s} ^{m-n} \Sigma_{t} ^n \MGL_X]= 
\MGL_X^{m, n} (X)$. The formalism of 
exact couples then yields us a spectral sequence 
\begin{equation}\label{eqn:pre-spectral-seq}
E_1 ^{p,q} (X, n) = E^{p,q} _1 \Rightarrow \MGL_X^{m, n} (X).
\end{equation}

We now have 
\begin{equation}\label{eqn:SS-1}
\begin{array}{lll}
E_1 ^{p,q} (X, n) & = & [\Sigma^{\infty}_TX_+, 
\Sigma_{s} ^{p+q-n} \Sigma_{t} ^n s_p \MGL_X] \\
& {\cong}^{1} &
[\Sigma^{\infty}_TX_+, 
\Sigma_{s} ^{p+q-n} \Sigma_{t} ^n s_p \mathbf L v^*\MGL] \\
& {\cong}^{2} &
[\Sigma^{\infty}_TX_+, 
\Sigma_{s} ^{p+q-n} \Sigma_{t} ^n \mathbf L v^*(s_p \MGL)] \\
& {\cong}^{3} &
[\Sigma^{\infty}_TX_+, 
\Sigma_{s} ^{p+q-n} \Sigma_{t} ^n \mathbf L v^*(\Sigma^p_T H(\bL^{-p}))] \\
& {\cong} &
[\Sigma^{\infty}_TX_+, 
\Sigma_{s} ^{p+q-n} \Sigma_{t} ^n \Sigma^p_T \mathbf L v^*(H(\bL^{-p}))].
\end{array}
\end{equation}
In this sequence of isomorphisms, ${\cong}^{1}$ is shown above,
${\cong}^{2}$ follows from \cite[Thm.~3.7]{Pelaez-3} and
${\cong}^{3}$ follows from the isomorphism
$s_p \MGL \xrightarrow{\cong} \Sigma^p_T H(\bL^{-p})$, as shown, for example,
in \cite[8.6, p.~46]{Hoyois}, where 
$\bL = \oplus_{i \le 0} \bL^i \cong \oplus_{i \ge 0} MU_{2i}$ is the Lazard ring.

Since $\bL$ is a torsion-free abelian group, it follows from 
~\ref{thm:cdhdescent-2} that the last term of
~\eqref{eqn:SS-1} is same as $H^{3p+q}(X, \Z(n+p)) \otimes_{\Z} \bL^{-p}$.

The spectral sequence ~\eqref{eqn:pre-spectral-seq} is actually 
identical to an $E_2$-spectral
sequence after reindexing. Indeed, letting $\wt{E}^{p',q'}_2 = 
H^{p'-q'}(X, \Z(n-q')) \otimes_{\Z} \bL^{q'}$ and
using ~\eqref{eqn:SS-1}, an elementary calculation shows that the invertible 
transformation $(3p+q, n+p) \mapsto (p'-q', n-q')$ 
yields 
\begin{equation}\label{eqn:SS-2}
E^{p+1, q}_1 \cong [\Sigma^{\infty}_TX_+, 
\Sigma_{s} ^{p+q+1-n} \Sigma_{t} ^n s_{p+1} \MGL_X] 
\end{equation}
\[
\hspace*{4.5cm}
\cong H^{(p'+2)-(q'-1)}(X, \Z(n- (q'-1)) \otimes_{\Z} \bL^{q'-1}
= \wt{E}^{p'+2, q'-1}_2.
\]

It is clear from ~\eqref{eqn:SS-0} that the $E_1$-differential of 
the above spectral sequence is $d^{p.q}_1: E^{p,q}_1 \to E^{p+1,q}_1$
and ~\eqref{eqn:SS-2} shows that this differential is identified with the
differential $d^{p',q'}_2 = d^{p,q}_1: \wt{E}^{p',q'}_2 \to \wt{E}^{p'+2, q'-1}_2$.
Inductively, it follows that the chain complex
$\{E^{p,q}_r \xrightarrow{d_r} E^{p+r,q-r+1}_r\}$ is transformed to the
chain complex $\{\wt{E}^{p',q'}_{r+1} \xrightarrow{d_{r}} 
\wt{E}^{p'+r+1, q'-r}_{r+1}\}$.
Combining this with ~\eqref{eqn:pull-back-0},
we conclude the following.

\begin{thm}\label{thm:Slice-SS-sing}
Let $k$ be a field which has characteristic zero
and let $X \in \Sch_k$. Then for any integer $n \in \Z$,
there is a strongly convergent spectral sequence:
\begin{equation}  \label{eqn:SS-MGL-0}
E^{p,q}_2 = H^{p-q}(X, \Z(n-q)) \otimes_{\Z} \bL^{q}
\Rightarrow \MGL^{p+q, n} (X).
\end{equation}
The differentials of this spectral sequence are given by
$d_r: E^{p,q}_r \to E^{p+r, q-r+1}_r$; and for every $p$, $q\in \mathbb Z$,
there exists $N>0$ such that $E^{p,q}_r=E^{p,q}_\infty$ for $r\geq N$,
where $E^{p,q}_\infty$
is the associated graded $gr^{-q}F^{\bullet}$ with respect
to the descending filtration on $\MGL^{p+q,n}(X)$ (see \ref{def.indfil}).
Furthermore, this spectral sequence degenerates with rational coefficients.
\end{thm}
\begin{proof}
The construction of the spectral sequence is shown above.
Since $\MGL$ is bounded by \cite[Thm.~8.12]{Hoyois}, it
follows from ~\ref{prop:bound-0} that
the spectral sequence ~\eqref{eqn:SS-MGL-0} is strongly convergent.
Thus, we deduce the existence of $N>0$ such that $E_r^{p,q}=E_\infty ^{p,q}$
for $r\geq N$.

For the degeneration with rational coefficients, we note that
the maps $f_{p}\MGL \to s_{p}\MGL  \cong \Sigma^p_T H(\bL^{-p})$ 
rationally split to yield an isomorphism of spectra
$\MGL_{\Q} \xrightarrow{\cong} \oplus_{p\ge 0} \Sigma _{T}^{p}H
(\bL^{-p}_{\Q})$
in $\stablehomotopy$ \cite[Thm. 10.5 and Cor. 10.6(i)]{Landwexact}.
The desired degeneration of the spectral sequence now follows
immediately from its construction above. 
\end{proof}


\begin{remk}  \label{rmk.MGLnoconv}
If $k$ is a perfect field of positive characteristic $p$, we observe that
$s_r(\MGL[\tfrac{1}{p}])\cong \Sigma^r_T H(\bL^{-r})[\tfrac{1}{p}]$ for
every $r\in \mathbb Z$ \cite[8.6, p.~46]{Hoyois}, so $s_r(\MGL[\tfrac{1}{p}])$
has a weak structure of smooth traces \cite[5.2.4]{Kelly}.  Thus, we can apply
\cite[4.2.29]{Kelly} to conclude that $\mathbf L v^{\ast} s_r(\MGL[\tfrac{1}{p}])
\cong s_r(\mathbf L v^\ast \MGL[\tfrac{1}{p}])$. Except for this 
identification, the proof of \thmref{thm:Slice-SS-sing} does not depend
on the characteristic of $k$. We thus obtain a spectral sequence as in
\eqref{eqn:SS-MGL-0}:
\[ E^{a,b}_2 = H^{a-b}(X, \Z(n-b)) \otimes_{\Z} \bL^{b}[\tfrac{1}{p}]
\Rightarrow \MGL^{a+b, n} (X)[\tfrac{1}{p}].
\]
But we can only guarantee strong convergence when $X\in \Sm _k$ 
\cite[Thm.~8.12]{Hoyois}.  In general, for $X\in \Sch _k$, the spectral
sequence satisfies the weaker convergence of
\ref{pr:findiff.cp}\eqref{pr:findiff.cp-1}-\eqref{pr:findiff.cp-2}.
In this case, the strong convergence will follow if one knew that $\MGL$
has a structure of traces.
\end{remk}

\subsection{The slice spectral sequence for $\KGL$}\label{sec:KH-theory}
For any noetherian scheme $X$ of finite Krull dimension, the motivic 
$T$-spectrum $\KGL_X \in Spt({\sM_X})$ was defined by Voevodsky
(see \cite[\S~6.2]{Voev-0}) which has the property that it represents
algebraic $K$-theory of objects in $\Sm_X$ if $X$ is regular.
It was later shown by Cisinski \cite{Cisinski} that for $X$ not necessarily 
regular, $\KGL_X$ represents Weibel's homotopy invariant $K$-theory 
$KH_*(Y)$ for $Y \in \Sm_X$.  Like $\MGL_X$, there is a structure of
symmetric spectrum on $\KGL_X$ for which we refer 
to \cite[p. 157, p. 176]{Jardine-ksym}.

Let $k$ be a field of exponential characteristic $p$.
The map $\mathbf L v^*(\KGL_k) \to \KGL_X$ is an isomorphism
by \cite[Prop.~3.8]{Cisinski}. It is known that 
$s_r \KGL_k \cong \Sigma^r_T H\Z$, for $r \in \Z$ 
(see \cite[Thm.~6.4.2]{Lev} if $k$ is perfect and \cite[\S~1, p.~1158]{RO}
in general). 
It follows from \cite[Thm.~3.7]{Pelaez-3} (in positive characteristic
we use instead \cite[4.2.29]{Kelly})
that $\mathbf L v^*(s_r \KGL [\tfrac{1}{p}]_k) \cong s_r (\mathbf L v^* \KGL [\tfrac{1}{p}]_k)
\cong  s_r \KGL [\tfrac{1}{p}]_X$.
One also knows that $(\KGL_k)_{\Q} \cong \oplus_{p\in \Z} \Sigma _T^p H\Q$
in $\stablehomotopy$ \cite[5.3.17 and 5.3.10]{Riou}.
We can thus use the Bott periodicity of $\KGL_X$ 
and repeat the construction of \S~\ref{sec:slicess}
mutatis mutandis (with $n = 0$) to conclude the following.

\begin{thm}\label{thm:KH-SS}
Let $k$ be a field that admits resolution of singularities (resp. a field of 
exponential characteristic $p>1$);
and let $X \in \Sch_k$. Then there is a strongly convergent spectral sequence
\begin{align}\label{eqn:KH-S-0}
E^{a,b}_2 = H^{a-b}(X, \Z (-b))
& \Rightarrow KH_{-a-b}(X) \\ \label{eqn:KH-S-p}
resp. \; \; E^{a,b}_2 = H^{a-b}(X, \Z (-b))\otimes _{\mathbb Z}\Z [\tfrac{1}{p}]
& \Rightarrow KH_{-a-b}(X) \otimes _{\mathbb Z}\Z [\tfrac{1}{p}].
\end{align}
The differentials of this spectral sequence are given by
$d_r: E^{a,b}_r \to E^{a+r, b-r+1}_r$; and for every $a$, $b\in \mathbb Z$,
there exists $N>0$ such that $E^{a,b}_r=E^{a,b}_\infty$ for $r\geq N$,
where $E^{a,b}_\infty$
is the associated graded $gr^{-b}F^{\bullet}$ with respect
to the descending filtration on $KH_{-a-b}(X)$ (resp. $KH[\frac{1}{p}]_{-a-b}(X)$),
see \ref{def.indfil}.
Furthermore, this spectral sequence degenerates with rational coefficients.
\end{thm}
\begin{proof}
If $k$ admits resolution of singularities, we just need to show that the spectral
sequence is convergent.  For this, we observe that $\KGL _k$ is the spectrum
associated to the Landweber exact $\mathbb L$-algebra $\mathbb Z [\beta, \beta ^{-1}]$
that classifies the multiplicative formal group law \cite[Thm.~1.2]{SpOst}.
Thus \cite[8.12]{Hoyois} implies that $\KGL _k$ is bounded with respect to
the slice filtration (this argument also applies in positive characteristic).  
Hence, the convergence follows from \ref{prop:bound-0}.

In the case of positive characteristic, the existence of the spectral sequence follows 
by combining the argument of \S~\ref{sec:slicess} with \ref{lem.compcoef1} and \ref{lem.compcoef2}.  
To establish the convergence, it suffices to check that $\KGL[\frac{1}{p}]_k$ satisfies 
the conditions in \ref{prop:bound-charp}.

We have already seen that $\KGL _k$ is bounded with respect to the slice filtration.
Thus, by \ref{lem.compcoef1}\eqref{b.coef1} we conclude that $\KGL[\frac{1}{p}]_k$ is bounded 
with respect to the slice filtration as well.   On the other hand, it follows from
\cite[5.2.3]{Kelly} that $\KGL [\frac{1}{p}]_k$ has a structure of traces in 
the sense of \cite[4.3.1]{Kelly}.  Finally, since $s_r \KGL_k \cong \Sigma^r_T H\Z$ for $r \in \Z$,
combining \cite[5.2.4]{Kelly} and \ref{lem.compcoef2}, we deduce that 
$s_r (\KGL [\frac{1}{p}])_k$
has a  weak structure of smooth traces in the sense of \cite[4.2.27]{Kelly}.  This
finishes the proof.
\end{proof}

\begin{remk}\label{remk:Haesemeyer}
When $k$ has characteristic zero,
the spectral sequence of \thmref{thm:KH-SS} is not new and was constructed
by Haesemeyer (see \cite[Thm.~7.3]{Haes}) using a different approach.
However, the expected degeneration (rationally) 
of this spectral sequence and its 
positive characteristic analogue are new.
\end{remk}

As a combination of ~\ref{thm:KH-SS} and \cite[Thm.~9.5, 9.6]{TT}, we obtain
the following spectral sequence for the algebraic $K$-theory $K^B(-)$ of 
singular schemes \cite{TT}.

\begin{cor}\label{cor:Alg-K-SS}
Let $k$ be a field of exponential characteristic $p > 1$. Let
$\ell \neq p$ be a prime and $m \ge 0$ any integer. 
Given any $X \in \Sch_k$, there exist strongly convergent spectral sequences
\begin{align}\label{eqn:K-SS-0}
E^{a,b}_2 = H^{a-b}(X, \Z (-b))\otimes _{\mathbb Z}\Z[\tfrac{1}{p}]
& \Rightarrow K^B_{-a-b}(X) \otimes _{\mathbb Z}\Z[\tfrac{1}{p}] ;\\
\label{eqn:K-SS-1}
E^{a,b}_2 = H^{a-b}(X, \Z/{\ell^m} (-b))
& \Rightarrow K^B/{\ell^m}_{-a-b}(X).
\end{align}
\end{cor}

\section{Applications I: Comparing cobordism, $K$-theory
and cohomology}\label{sec:Appl}
In this section, we deduce some geometric applications of the slice
spectral sequences for singular schemes. More applications will appear
in the subsequent sections.

Consider the edge map $\MGL = f_0\MGL \to s_0\MGL \cong H\Z$ in the spectral
sequence \eqref{eqn:SS-MGL-0}.  This induces a natural
map $\nu _{X}: \MGL^{i,j}(X) \to H^{i}(X, \Z(j))$ for every $X \in \Sch_k$ and
$i,j \in \Z$.

The following result shows that there is no distinction between
algebraic cycles and cobordism cycles at the level of 0-cycles.

\begin{thm}  \label{thm:MGL-MC-iso}
Let $k$ be a field which admits resolution of singularities (resp.
a perfect field of positive characteristic $p$). Then
for any $X \in \Sch_k$ of dimension $d$, we have
$H^{2a-b}(X, \Z(a)) = 0$ (resp. $H^{2a-b}(X, \Z(a))\otimes _{\mathbb Z}\mathbb 
Z[\tfrac{1}{p}] = 0$) whenever $a > d+b$. In particular, for every
$X\in \Sch _k$ (resp. $X\in \Sm _k$), the map 
\begin{align}  \label{eqn:MGL-MC-iso-0}
\nu_X: \MGL^{2d+i,d+i}(X) \to H^{2d+i}(X, \Z(d+i))\\ \label{eqn:MGL-MC-iso-1}
resp. \; \;  \nu_X: \MGL^{2d+i,d+i}(X) \otimes _{\mathbb Z}\mathbb Z [\tfrac{1}{p}] 
\to H^{2d+i}(X, \Z(d+i))\otimes _{\mathbb Z} \mathbb Z [\tfrac{1}{p}]
\end{align}
is an isomorphism for all $i \ge 0$.
\end{thm}
\begin{proof}
Using the spectral sequence ~\eqref{eqn:SS-MGL-0} (resp. \ref{rmk.MGLnoconv})
and the fact that $\bL^{> 0} = 0$, the isomorphism of ~\eqref{eqn:MGL-MC-iso-0}
(resp. \ref{eqn:MGL-MC-iso-1})
follows immediately from the vanishing assertion for the motivic cohomology.

To prove the vanishing result, we note that for $X \in \Sm_k$, 
there is an isomorphism 
\linebreak
$H^{2a-b}(X, \Z(a)) \cong \CH^a(X, b)$
by \cite{Voev-1}, and the latter group is clearly zero if $a > d+b$
by definition of Bloch's higher Chow groups.

If $X$ is not smooth and $k$ admits resolution of singularities, 
our assumption on $k$ implies that 
there exists a $cdh$-cover 
$\{X' \amalg Z \rightarrow X\}$ of $X$, such that
$X' \in \Sm_k$, $\dim(Z) < \dim(X)$ and $\dim(W) < \dim(X)$, where  
we set $W = X' \times _{X} Z$. The $cdh$-descent for the motivic cohomology
yields an exact sequence
\[
H^{2a-b-1}(W, \Z(a)) \xrightarrow{\partial}  H^{2a-b}(X, \Z(a)) 
\to H^{2a-b}(X', \Z(a)) \oplus H^{2a-b}(Z, \Z(a)).
\]
The smooth case of our vanishing result shown above and an induction on the
dimension together imply that the two end terms of this exact sequence
vanish. Hence, the middle term vanishes too.

If $X$ is not smooth and $k$ is perfect of positive characteristic, 
we argue as in
\ref{prop:bound-charp}.  Namely,
by a theorem of Gabber \cite[Introduction Thm. 3(1)]{Gabber}
and Temkin's strengthening \cite[Thm. 1.2.9]{Temkin} of Gabber's result,
there exists $W\in \Sm _k$
and a surjective proper map $h:W\rightarrow X$, which is generically \'etale of degree
$p^r$, $r\geq 1$.  Then by a theorem of Raynaud-Gruson
\cite[Thm. 5.2.2]{Ray}, there exists a blow-up $g:X'\rightarrow X$ with center $Z$
such that the following diagram commutes, where $h'$ is finite flat surjective 
of degree $p^r$
and $g':W'\rightarrow W$ is the blow-up of W with center $h^{-1}(Z)$:
\begin{align}  \label{bwupsqr.1}
\begin{array}{c}
\xymatrix{W' \ar[d]_-{g'} \ar[r]^-{h'}& X' \ar[d]^-{g} \\
	W \ar[r]_-{h} & X.}
\end{array}
\end{align}

Thus we have a $cdh$-cover 
$\{X' \amalg Z \rightarrow X\}$ of $X$, such that $\dim_k(Z) < \dim_k(X)$ and 
$\dim_k(E) < \dim_k(X)$, where we set $E = X' \times _{X} Z$.
Then by $cdh$-excision, the following diagram is exact:
\begin{align*}
H^{2a-b-1}(E,\mathbb Z (a))\otimes _{\mathbb Z}\mathbb Z [\tfrac{1}{p}] \rightarrow 
H^{2a-b}(X,\mathbb Z (a))\otimes _{\mathbb Z}\mathbb Z [\tfrac{1}{p}] \rightarrow \\
H^{2a-b}(X',\mathbb Z (a))\otimes _{\mathbb Z}\mathbb Z [\tfrac{1}{p}] \oplus 
H^{2a-b}(Z,\mathbb Z (a))\otimes _{\mathbb Z}\mathbb Z [\tfrac{1}{p}].
\end{align*}

By induction on the dimension, this reduces to the following exact sequence:
\[ \xymatrix{0 \ar[r]& H^{2a-b}(X,\mathbb Z (a))\otimes _{\mathbb Z}\mathbb Z [\tfrac{1}{p}]
\ar[r]^-{g^\ast} &  
H^{2a-b}(X',\mathbb Z (a))\otimes _{\mathbb Z}\mathbb Z [\tfrac{1}{p}]}.
\]
So it suffices to show that $g^\ast =0$.  In order to prove this, we observe
that \eqref{bwupsqr.1} commutes, therefore since $W\in \Sm _k$:
$H^{2a-b}(W,\mathbb Z (a))\otimes _{\mathbb Z}\mathbb Z [\tfrac{1}{p}]=0$; and we conclude that 
$h^{\prime \ast} \circ g^\ast = g^{\prime \ast }\circ h^\ast=0$.
Thus, it is enough to see that
\[ h^{\prime \ast}: H^{2a-b}(X',\mathbb Z (a))\otimes _{\mathbb Z}\mathbb Z [\tfrac{1}{p}] \rightarrow
H^{2a-b}(W',\mathbb Z (a))\otimes _{\mathbb Z}\mathbb Z [\tfrac{1}{p}]
\]
is injective.  Let $v':X'\rightarrow \Spec k$, and $\epsilon: \mathbf L v^{\prime \ast}H\mathbb Z [\tfrac{1}{p}]
\rightarrow \mathbf R h'_{\ast} \mathbf L h^{\prime \ast} \mathbf L v^{\prime \ast}H\mathbb Z [\tfrac{1}{p}]$
be the map given by the unit of the adjunction $(\mathbf L h^{\prime \ast}, \mathbf R h'_{\ast})$.
By the naturality of the isomorphism in 2.13, we deduce that $h^{\prime \ast}$ gets identified
with the map induced by $\epsilon$ (see \ref{thm:cdhdescent-2}):
\[ \epsilon _{\ast}: \Hom _{\stablehomotopy _{X'}}(\susp{m,n}\Sigma^{\infty}_T(X'_{+}), 
\mathbf L v^{\prime \ast}H\mathbb Z [\tfrac{1}{p}]) \rightarrow
\Hom _{\stablehomotopy _{X'}}(\susp{m,n}\Sigma^{\infty}_T(X'_{+}), 
\mathbf R h'_{\ast} \mathbf L h^{\prime \ast} \mathbf L v^{\prime \ast}H\mathbb Z 
[\tfrac{1}{p}]).
\]
By \cite[5.2.4]{Kelly}, $H\mathbb Z [\tfrac{1}{p}]$ has a structure of traces  in the sense of
\cite[4.3.1]{Kelly}.  Thus, we deduce from \cite[4.3.1(Deg) p. 101]{Kelly} that $\epsilon _{\ast}$
is injective since $h'$ is finite flat surjective of degree $p^r$.
This finishes the proof.
\end{proof}

\begin{remk}  \label{degcit}
For $X \in \Sm_k$ and $i=0$, 
the isomorphism of ~\eqref{eqn:MGL-MC-iso-0} was proved by 
D\'eglise \cite[Cor.~4.3.4]{Deglise1}.
\end{remk}

When $A$ is a field, the following result was proven by Morel using
methods of unstable motivic homotopy theory \cite[p. 9 Cor. 1.25]{Morel}.  
Taking for granted the result for fields, D\'eglise \cite{Deglise1} proved
~\ref{thm:MGL-Milnor} using homotopy modules.  Spitzweck proved
~\ref{thm:MGL-Milnor}  for localizations of a Dedekind domain
in \cite[Cor. 7.3]{Sp-2}.

\begin{thm}\label{thm:MGL-Milnor}
Let $k$ be a perfect field of exponential characteristic $p$. Then
for any regular semi-local ring $A$ which is essentially of finite type over 
$k$, and for any integer $n \ge 0$, the map 
\begin{equation}\label{eqn:MGL-Milnor-0}
\MGL^{n,n}(A)\otimes _{\mathbb Z}\mathbb Z[\tfrac{1}{p}] \to H^{n}(A, \Z(n))
\otimes _{\mathbb Z} \mathbb Z [\tfrac{1}{p}]. 
\end{equation}
is an isomorphism. In particular, there is a natural isomorphism
$\MGL^{n,n}(A)\otimes _{\mathbb Z}\mathbb Z [\tfrac{1}{p}] \cong K^M_n(A)
\otimes _{\mathbb Z} \mathbb Z [\tfrac{1}{p}]$ if $k$ is also infinite.
\end{thm}
\begin{proof}
Using the spectral sequence ~\eqref{eqn:SS-MGL-0} and the fact that
$\bL^{> 0} = 0$, it suffices to prove that 
$E^{n+i+j, -i}_2(A) = 0$ for every $j \ge 0$ and $i \ge 1$.  In positive
characteristic, we can use \ref{rmk.MGLnoconv} since $A$ is regular.
Notice that \eqref{eqn:SS-MGL-0}-\eqref{rmk.MGLnoconv}
are strongly convergent for $A$ by \cite[8.9-8.10]{Hoyois}.

On the one hand, we have isomorphisms
\begin{align*}
E^{n+i+j, -i}_2(A) & =
H^{n+2i+j}(A, \Z(n+i))\otimes _{\mathbb Z} \mathbb Z [\tfrac{1}{p}] \\
& \cong 
\CH^{n+i}(A, 2n+2i-n-2i-j) \otimes _{\mathbb Z} \mathbb Z [\tfrac{1}{p}]
= \CH^{n+i}(A, n-j) \otimes _{\mathbb Z} \mathbb Z [\tfrac{1}{p}].
\end{align*}

On the other hand, letting $F$ denote the fraction field of $A$,
the Gersten resolution for the higher Chow groups
(see \cite[Thm.~10.1]{Bloch}) shows that the restriction map
$\CH^{n+i}(A, n-j) \to \CH^{n+i}(F, n-j)$ is injective.
But the term $\CH^{n+i}(F, n-j)$ is zero whenever $j \ge 0, i \ge 1$
for dimensional reasons. We conclude that $E^{n+i+j, -i}_2(A) = 0$.
The last assertion of the theorem now follows from 
the isomorphism $\CH^n(A, n) \cong K^M_n(A)$ by
\cite[Thm.~1.1]{Kerz}.
\end{proof}

\subsection{Connective $K$-theory}\label{sec:Conn-K}
Let $k$ be a field of exponential characteristic $p$ and let
$X \in \Sch_k$. Recall that the {\sl connective $K$-theory} spectrum
$\KGL^0_X$ is defined to be the motivic $T$-spectrum $f_0\KGL_X$
in $\stablehomotopy_X$ (see ~\eqref{eqn:Slice-triangle}).  Strictly speaking,
$\KGL^0_X$ should be called effective $K$-theory, nevertheless we will
follow the terminology of Dai-Levine \cite{DL}.

In particular, there is a canonical map $u_X: \KGL^0_X \to \KGL_X$
which is universal for morphisms from objects of 
$\stablehomotopy _{X}^{\fnteff}$ to $\KGL_X$. 
For any $Y \in \Sm_X$, we let $CKH^{p,q}(Y) =
\Hom_{\stablehomotopy_X}(\Sigma^{\infty}_TY_+, 
\Sigma^{p,q} \KGL^0_X)$.
Using an analogue of ~\ref{thm:Slice-SS-sing} for $\KGL^0_X$, one can prove
the existence of the cycle class map for the higher Chow groups as
follows.

\begin{thm}\label{thm:Conn-KGL}
Let $k$ be a field of exponential characteristic $p$ and let
$X \in \Sch_k$ have dimension $d$. Then the map
$\KGL^0_X [\tfrac{1}{p}]\to s_0\KGL_X [\tfrac{1}{p}] \cong H\Z [\tfrac{1}{p}]$ induces for 
every integer $i \ge 0$, an isomorphism
\begin{equation}\label{eqn:Conn-KGL-0}
CKH^{2d+i, d+i}(X)\otimes _{\mathbb Z} \mathbb Z [\tfrac{1}{p}]
 \xrightarrow{\cong} H^{2d+i}(X, \Z(d+i)) \otimes _{\mathbb Z} \mathbb Z [\tfrac{1}{p}].
\end{equation}
In particular, the canonical map $\KGL^0_X \to \KGL_X$ 
induces a natural cycle class map
\begin{equation}\label{eqn:Conn-cycle}
cyc_i: H^{2d+i}(X, \Z(d+i))\otimes _{\mathbb Z}\mathbb Z [\tfrac{1}{p}] \to KH_i(X)
\otimes _{\mathbb Z}\mathbb Z [\tfrac{1}{p}].
\end{equation}
\end{thm}
\begin{proof}
First we assume that $k$ admits resolution of singularities.
It follows from the definition that $\KGL^0_X$ is a connective
$T$-spectrum and we have $\mathbf L v^*(\KGL^0_k) \xrightarrow{\cong}
\KGL^0_X$ by \cite[Thm.~3.7]{Pelaez-3}. 
One also knows that $s_r\KGL^0_k \cong \Sigma^r_T H\Z$ for
$r \ge 0$ \cite[Thm.~6.4.2]{Lev} and zero otherwise.  
The proof of ~\ref{thm:Slice-SS-sing} can now be repeated 
verbatim to conclude that for each $n \in \Z$, there is a strongly convergent 
spectral sequence
\begin{equation}\label{eqn:Conn-KGL-1}
E^{a,b}_2 = H^{a-b}(X, \Z(n-b)) \otimes_{\Z} \Z_{b\le0}
\Rightarrow CKH^{a+b, n}(X),
\end{equation}
where $\Z_{b\le0} = \Z$ if $b \le  0$ and is zero otherwise.
Furthermore, this spectral sequence degenerates with rational coefficients.

One now repeats the proof of ~\ref{thm:MGL-MC-iso} to conclude that the
edge map $CKH^{2d+i, d+i}(X) \to H^{2d+i}(X, \Z(d+i))$ is an isomorphism for
every $i \ge 0$.
Finally, we compose the inverse of this isomorphism with canonical map
$CKH^{2d+i, d+i}(X) \to KH_{i}(X)$ to get the desired cycle class map.

If the characteristic of $k$ is positive, then 
$s_r(\KGL ^0 _k)\cong \Sigma^r_T H\mathbb Z$
for every $r\geq 0$ and zero otherwise \cite[Thm.~6.4.2]{Lev}.  So $s_r(\KGL ^0 _k[\tfrac{1}{p}])$
has a weak structure of traces \cite[5.2.4]{Kelly}. By
\ref{lem.compcoef2}, we deduce that
$s_r(\KGL ^0 _k[\tfrac{1}{p}])\cong \Sigma^r_T H\mathbb Z [\tfrac{1}{p}]$ for
every $r\geq 0$ and zero otherwise.
Thus, we can apply
\cite[4.2.29]{Kelly} to conclude that $\mathbf L v^{\ast} (\KGL ^0 _k [\tfrac{1}{p}])
\cong \KGL ^0 _X [\tfrac{1}{p}]$.  Then the argument of ~\ref{thm:KH-SS}
applies, and we conclude that for each $n \in \Z$, 
there is a strongly convergent spectral sequence
\begin{equation}\label{eqn:Conn-KGL-p}
E^{a,b}_2 = H^{a-b}(X, \Z(n-b)) \otimes_{\Z} \Z [\tfrac{1}{p}]_{b\le0}
\Rightarrow CKH^{a+b, n}(X) \otimes_{\Z} \Z [\tfrac{1}{p}],
\end{equation}
By ~\ref{thm:MGL-MC-iso}, $H^{2a-b}(X, \Z(a))\otimes _{\mathbb Z}\mathbb 
Z [\frac{1}{p}] = 0$ whenever $a > d+b$.  Thus, combining the
spectral sequence ~\eqref{eqn:Conn-KGL-p} and the fact that $\bL^{> 0} = 0$, 
we deduce the isomorphism of ~\eqref{eqn:Conn-KGL-0} with 
$\mathbb Z [\frac{1}{p}]$-coefficients:
\[  CKH^{2d+i, d+i}(X)\otimes _{\mathbb Z} \mathbb Z [\tfrac{1}{p}]
 \xrightarrow{\cong} H^{2d+i}(X, \Z(d+i)) \otimes _{\mathbb Z} \mathbb Z [\tfrac{1}{p}].
\]
\end{proof}

An argument identical to the proof of ~\ref{thm:MGL-Milnor}
shows that for any regular semi-local ring $A$ 
which is essentially of finite type over an infinite field $k$ and any 
integer $n \ge 0$,
there is a natural isomorphism (notice that in positive characteristic,
the spectral sequence is also strongly convergent integrally since $A$ is regular):
\begin{equation}\label{eqn:Conn-Milnor}
CKH^{n,n}(A) \xrightarrow{\cong} 
K^M_n(A).
\end{equation}
 
Moreover, the canonical map $CKH^{n,n}(A) \to 
K_n(A)$ respects products
\cite[Thm.~3.6.9]{Pelaez-1}, hence it
coincides with the 
known map $K^M_n(A) \to 
K_n(A)$.
This shows that the Milnor $K$-theory is represented by the connective
$K$-theory and one gets a lifting of the relation between the Milnor
and Quillen $K$-theory of smooth semi-local schemes to the level of 
$\stablehomotopy$.  In particular, it is possible to recover Milnor $K$-theory
and its map into Quillen $K$-theory from the $T$-spectrum $KGL$ (which 
represents Quillen $K$-theory in $\stablehomotopy$ for smooth
$k$-schemes) by
passing to its $(-1)$-effective cover $f_0 KGL_k\rightarrow KGL_k$.

As another consequence of the slice spectral sequence, one gets the
following comparison result between the connective and non-connective
versions of the homotopy $K$-theory.
The homological analogue of this results was shown in
\cite[Cor.~5.5]{DL}.

\begin{thm}\label{thm:Conn-Non-conn}
Let $k$ be a field of exponential characteristic $p$ and let
$X \in \Sch_k$ have dimension $d$. Then the canonical map
$CKH^{2n, n}(X)\otimes _{\mathbb Z}\mathbb Z [\tfrac{1}{p}] \to 
KH_0(X) \otimes _{\mathbb Z}\mathbb Z [\tfrac{1}{p}]$ is an isomorphism 
for every integer $n \le 0$.
\end{thm}
\begin{proof}
If $k$ admits resolution of singularities, we
observe that the slice spectral sequence is functorial for morphisms
of motivic $T$-spectra. Since $H^{2q}(X, \Z(q)) = 0$ for $q < 0$, 
a comparison of the spectral sequences 
~\eqref{eqn:KH-S-0} and ~\eqref{eqn:Conn-KGL-1} shows that it is enough to 
prove that for every $r \ge 2$ and $q \le 0$, either
$q+r-1 \le 0$ or $H^{-q-r - (q+r-1)}(X, \Z(1-r-q)) =
H^{1-2r-2q}(X, \Z(1-r-q)) = 0$.
But this is true because $H^{1-2r-2q}(X, \Z(s)) = 0$ if $s < 0$.

In positive characteristic, we use the same argument as above for the spectral sequences
\eqref{eqn:KH-S-p} and \eqref{eqn:Conn-KGL-p} to deduce that
\[  CKH^{2n, n}(X)\otimes _{\mathbb Z}\mathbb Z _{(\ell)} \to 
KH_0(X) \otimes _{\mathbb Z}\mathbb Z _{(\ell)}
\]
is an isomorphism for every prime $\ell \neq p$.  Then, the result follows from
\ref{lem.compcoef1}\eqref{b.coef1}.
\end{proof}

Yet another consequence of the above spectral sequences
is the following direct 
verification of Weibel's vanishing conjecture for negative
$KH$-theory and negative $CKH$-theory of singular schemes.
For $KH$-theory, there are other proofs of this conjecture by
Haesemeyer \cite[Thm.~7.1]{Haes} in characteristic zero and
by Kelly \cite[Thm.~3.5]{Kelly-2} and Kerz-Strunk \cite{KSt}
in positive characteristic using 
different methods.  We refer the reader to \cite{Cisinski},
\cite{CHW1}, \cite{GH-nK}, \cite{Kerz-2},
\cite{Krishna-nK}  and  
\cite{Weibel2} for more results associated to Weibel's conjecture.
The vanishing result below for $CKH$-theory is new in any characteristic.

\begin{thm}\label{thm:Vanishing-conj}
Let $k$ be a field of exponential characteristic $p$ and let
$X \in \Sch_k$ have dimension $d$. Then $CKH^{m,n}(X)\otimes _{\mathbb Z}
\mathbb Z [\tfrac{1}{p}] = KH_{2n-m}(X) \otimes _{\mathbb Z}
\mathbb Z [\tfrac{1}{p}] = 0$ whenever $2n-m < -d$ and $KH_{-d}(X)
\otimes _{\mathbb Z} \mathbb Z [\tfrac{1}{p}] \cong H^d_{cdh}(X, \Z)
\otimes _{\mathbb Z} \mathbb Z [\tfrac{1}{p}]$.
\end{thm}
\begin{proof}
When $k$ admits resolution of singularities, 
using the spectral sequences ~\eqref{eqn:KH-S-0} and
~\eqref{eqn:Conn-KGL-0}, it suffices to show
that $H^{p-q}(X, \Z(n-q)) = 0$ whenever $2n-p-q+d < 0$.

If $n-q < 0$, then we already know that this motivic cohomology group is zero.
So we can assume $n-q \ge 0$.
We set $a = n-q, b = 2n-p-q$ so that $2a-b = 2n-2q-2n+p+q = p-q$.
Since $2n-p-q+d < 0$ and $n-q \ge 0$ by our assumption, we get 
\[
b+d-a = 2n-p-q+d-n+q = n-p+d = (2n-p-q+d)-(n-q) < 0.
\]
 
The theorem now follows because we have shown in the proof of 
~\ref{thm:MGL-MC-iso} that 
\linebreak
$H^{p-q}(X, \Z(n-q)) = H^{2a-b}(X, \Z(a)) =0$ as $a > b+d$.
This argument also shows that
$KH_{-d}(X) \cong H^d(X, \Z(0)) \cong H ^d_{cdh}(X, \Z)$.

In positive characteristic, we use the same argument with the spectral sequences
~\eqref{eqn:KH-S-p} and ~\eqref{eqn:Conn-KGL-p} to deduce that
$CKH^{m,n}(X)\otimes _{\mathbb Z}
\mathbb Z [\frac{1}{p}] = KH_{2n-m}(X) \otimes _{\mathbb Z}
\mathbb Z [\frac{1}{p}] = 0$ whenever $2n-m < -d$ and $KH_{-d}(X)
\otimes _{\mathbb Z} \mathbb Z [\frac{1}{p}] \cong H^d_{cdh}(X, \Z)
\otimes _{\mathbb Z} \mathbb Z [\frac{1}{p}]$.
\end{proof}

Weibel's conjecture on the vanishing of
certain negative $K$-theory was proven (after inverting the characteristic)
by Kelly \cite{Kelly-2}. Using our spectral sequence (which uses the
methods of \cite{Kelly}), 
we can obtain the following result (which follows as well 
from \cite{Kelly-2} via the cdh-descent spectral sequence). 
The characteristic zero version of this computation 
was proven in \cite[Thm.~02(1)]{CHW}, and  for arbitrary noetherian schemes,
we refer the reader to \cite[Cor. D]{Kerz-2}.

\begin{cor}\label{cor:KB-neg-vanish}
Let $k$ be a field of exponential characteristic $p$ and let
$X \in \Sch_k$ have dimension $d$. Then
\[ K^B_{-d}(X) \otimes _{\mathbb Z} \mathbb Z [\tfrac{1}{p}] \cong 
H^d_{cdh}(X, \Z) \otimes _{\mathbb Z} \mathbb Z [\tfrac{1}{p}].
\]
\end{cor}

\section{The Chern classes on $KH$-theory}\label{sec:Chern-classes}
In order to obtain more applications of the slice spectral sequence for 
$KH$-theory and the cycle class map (see ~\ref{thm:Conn-KGL}), we need to 
have a theory of Chern classes on the $KH$-theory of singular schemes.

Gillet \cite{Gillet} showed that any cohomology theory 
satisfying the projective bundle formula and some other
standard admissibility axioms admits a theory of Chern classes from
algebraic $K$-theory of schemes over a field. These Chern
classes are very powerful tools for understanding algebraic
$K$-theory groups in terms of various cohomology theories such as
motivic cohomology and Hodge theory. 
The Chern classes in Deligne cohomology are used to
define various regulator maps on $K$-theory and they also give rise to
the construction of intermediate Jacobians of smooth projective varieties
over $\C$.

If $k$ is a perfect field of exponential characteristic $p \ge 1$,
then Kelly \cite[Th.~5.5.10]{Kelly} has shown that the motivic cohomology
functor $X \mapsto \{H^i(X, \Z(j))[\tfrac{1}{p}]\}_{i,j \in \Z}$
satisfies the projective bundle formula in $\Sch_k$. 
This implies in particular by
Gillet's theory that there are functorial Chern class maps
\begin{equation}\label{eqn:Ch-K-mot}
c_{i,j}: K_j(X) \to H^{2i-j}(X, \Z(i))[\tfrac{1}{p}].
\end{equation}

In this section, we show that in characteristic zero,
Gillet's technique can be used to
construct the above
Chern classes on the homotopy invariant $K$-theory of singular schemes. 
Applications of these Chern classes to the understanding of 
the motivic cohomology and $KH$-theory of singular schemes 
will be given in the following two sections.

Let $k$ be a field of characteristic zero and let $\Sch_{{\Zar}/k}$ denote the 
category of separated schemes of finite type over $k$ equipped with the  Zariski 
topology. Let $\Sm_{{\Zar}/k}$ denote the full subcategory of smooth schemes 
over $k$ equipped with the  Zariski topology. 
For any $X \in \Sch_k$, let $X_{\Zar}$ denote small Zariski site of $X$.
A presheaf of spectra on $\Sch_k$ or $\Sm_k$ will mean a presheaf of
$S^1$-spectra.

Let $Pre(\Sch_{{\Zar}/k})$ be the category of presheaves of simplicial sets on 
$\Sch_{{\Zar}/k}$ equipped with the injective
Zariski local model structure, i.e., the weak 
equivalences are the maps that induce a weak equivalence of simplicial sets
at every Zariski stalk and the cofibrations are given by monomorphisms.
This model structure restricts to a similar model structure on the 
category $Pre(X_{\Zar})$ of presheaves of simplicial sets on $X_{\Zar}$ for every
$X \in \Sch_k$. 
We will write $\sH^{\rm big}_{\rm Zar}(k)$, $\sH^{\rm sml}_{\Zar}(X)$ for the 
homotopy categories
of $Pre(\Sch_{{\Zar}/k})$ and $Pre(X_{\Zar})$, respectively.

\subsection{Chern classes from $KH$-theory to motivic 
cohomology}\label{sec:Gill}
For any $X \in \Sch_k$, let
$\Omega BQP(X)$ denote the simplicial set obtained by taking the
loop space of the nerve of the category $QP(X)$ obtained by applying
Quillen's $Q$-construction to the exact category of locally free
sheaves on $X_{\Zar}$. Let $\sK$ denote the presheaf of simplicial sets on
$\Sch_{{\Zar}/k}$ given by $X \mapsto \Omega BQP(X)$. 
One knows that $\sK$ is a presheaf of infinite loop spaces so that there
is a presheaf of spectra $\wt{\sK}$ on $\Sch_k$ such that
$\sK = (\wt{\sK})_0$. Let $\wt{\sK}^B$ denote the Thomason-Trobaugh
presheaf of spectra on $\Sch_k$ such that $\wt{\sK}^B(X) = K^B(X)$
for every $X \in \Sch_k$. There is a natural map of presheaves of
spectra $\wt{\sK} \to \wt{\sK}^B$ which induces isomorphism between the
non-negative homotopy group presheaves.

Recall from \cite[Thm.~2.34]{Jardine-1} that the category of presheaves of
spectra on $\Sch_{{\Zar}/k}$ has a closed model structure where the
weak equivalences are given by the stalk-wise stable equivalence of
spectra and a map $f:E \to F$ is a cofibration if $f_0$ is a 
monomorphism and $E_{n+1} \amalg_{S^1 \wedge E_n} S^1 \wedge F_n \to F_{n+1}$
is a monomorphism for each $n \ge 0$. Let $\sH^s_{\Zar}(k)$ denote
the associated homotopy category.
There is a functor $\Sigma _s ^{\infty}: \sH^{\rm big}_{\rm Zar}(k) \to 
\sH^s_{\Zar}(k)$ which has a right adjoint.
We can consider the above model structure and the corresponding 
homotopy categories with respect to the Nisnevich and $cdh$-sites as well.

Let $\wt{\sK}_{cdh} \to \wt{\sK}^B_{cdh}$ denote the map between the
functorial  fibrant replacements in the above model structure on presheaves of
spectra on $\Sch_k$ with respect to the $cdh$-topology.  
Let $KH$ denote the presheaf of spectra on $\Sch_k$ such that
$KH(X)$ is Weibel's homotopy invariant $K$-theory of $X$
(see \cite{Weibel}).

The following is a direct consequence of the main result of \cite{Haes}.

\begin{lem}\label{lem:Elem}
Let $k$ be a field of characteristic zero.
For every $X \in \Sch_k$ and integer $p \in \Z$, there is a natural isomorphism
$KH_p(X) \xrightarrow{\cong} \H^{-p}_{cdh}(X, \sK_{cdh})$.
\end{lem}
\begin{proof}
We have a natural isomorphism 
\begin{equation}\label{eqn:Elem-0}
\begin{array}{lll}
\pi_p(\wt{\sK}_{cdh}(X)) & = & 
\Hom_{\sH^s_{cdh}(k)}(\Sigma _s^{\infty}(S^p_s \wedge X), \wt{\sK}) \\
& \cong &  \Hom_{\sH_{cdh}(k)}(S^p_s \wedge X, \sK) \\
& \cong &  \H^{-p}_{cdh}(X, \sK_{cdh}).
\end{array} 
\end{equation}

It is well known that the natural maps $K_p(X) \to
\pi_p(\wt{\sK}_{cdh}(X)) \to \pi_p(\wt{\sK}^B_{cdh}(X))$ are isomorphisms
for all $p \in \Z$ when $X$ is smooth over $k$.
In general, let $X \in \Sch_k$. We can find a Cartesian square
\begin{equation}\label{eqn:Elem-1} 
  \begin{array}{c}
\xymatrix@C1pc{
Z' \ar[r] \ar[d] & X' \ar[d]^{f} \\
Z \ar[r] & X,} 
  \end{array}
\end{equation}
where $X' \in \Sm_k$ and $f$ is a proper birational morphism 
which is an isomorphism
outside the closed immersion $Z \inj X$. An induction on dimension of $X$
and $cdh$-descent for $\wt{\sK}_{cdh}$ as well as $\wt{\sK}^B_{cdh}$ now show
that the map $\pi_p(\wt{\sK}_{cdh}(X)) \to \pi_p(\wt{\sK}^B_{cdh}(X))$ is an 
isomorphism for all $p \in \Z$. 
Composing the inverse of this isomorphism with the map in ~\eqref{eqn:Elem-0}, 
we get a natural isomorphism $\pi_p(\wt{\sK}^B_{cdh}(X))
\xrightarrow{\cong} \H^{-p}_{cdh}(X, \sK_{cdh})$.

On the other hand, it follows from \cite[Thm.~6.4]{Haes}
that the natural map $KH(X) \to \wt{\sK}^B_{cdh}(X)$ is a homotopy equivalence. 
We conclude that there is a natural isomorphism
$\nu_X: KH_p(X) \xrightarrow{\cong} \H^{-p}_{cdh}(X, \sK_{cdh})$
for every $X \in \Sch_k$ and $p \in \Z$. 
\end{proof}

Let $\sB \sG \sL$ denote the simplicial presheaf on $\Sch_k$
given by $\sB \sG \sL(X) = \colim_n BGL_n(\sO(X))$.
It is known (see \cite[Prop.~2.15]{Gillet}) 
that there is a natural section-wise weak equivalence
$\sK|_X \xrightarrow{\cong} \Z \times \Z_{\infty} {\sB \sG \sL}|_X$
in $Pre(\Sch_{{\Zar}/k})$ (see \S~\ref{sec:Gill}),
where $\Z_{\infty}(-)$ is the $\Z$-completion functor of Bousfield-Kan.

To simplify the notation, for any integer $q \in \Z$, we will write $\Gamma(q)$ 
for the presheaf on $\Sch_{{\Zar}/k}$ given by (see ~\S~\ref{sec:MCS}):
\[
\Gamma(q)(U) = \left\{ \begin{array}{ll}
\underline{C}_{\ast}z_{equi}(\A^q_k,0)(U)[-2q] & \mbox{if $q \ge 0$} \\
0 & \mbox{if $q < 0$}.
\end{array}
\right. 
\]

It is known that the restriction of $\Gamma(q)$ 
on $\Sm_{{\Zar}/k}$ is a sheaf (see, for instance, \cite[Def.~16.1]{MVW}).
We let $\Gamma(q)[2q] \to \sK(\Gamma(q), 2q)$ denote a functorial fibrant
replacement of $\Gamma(q)[2q]$ with respect to the 
injective Zariski local model
structure.

It follows from \cite[Ex.~3.1]{AS} that $\sK(\Gamma(q), 2q)$ is
a cohomology theory on $\Sm_{{\Zar}/k}$ satisfying all conditions of
\cite[Defs.~1.1-1.2]{Gillet}. We conclude from 
Gillet's construction (see \cite[\S 2, p. 225]{Gillet}) that for any
$X \in \Sm_{{\Zar}/k}$, there is a morphism of simplicial presheaves 
$C_q:{\sB \sG \sL}|_X \to \sK(\Gamma(q), 2q)|_X$ in $\sH^{\rm sml}_{\Zar}(X)$ 
which is natural in $X$.
Composing this with $\sK|_X \xrightarrow{\cong} \Z \times \Z_{\infty} 
{\sB \sG \sL}|_X$ and using the isomorphism
$\Z_{\infty}\sK(\Gamma(q), 2q) \cong \sK(\Gamma(q), 2q)$, we obtain a map
\[
C_q: \sK|_X \xrightarrow{\cong} \Z \times \Z_{\infty} {\sB \sG \sL}|_X \to
\Z \times \sK(\Gamma(q), 2q)|_X \to \sK(\Gamma(q), 2q)|_X
\]
in $\sH^{\rm sml}_{\Zar}(X)$, where the last arrow is the projection.

Since $\sK(\Gamma(q), 2q)$ is fibrant
in $Pre(\Sch_{{\Zar}/k})$, it follows from 
\cite[Cor. 5.26]{Jardine-lh} that the restriction$\sK(\Gamma(q), 2q)|_X$ 
is fibrant in $Pre(X_{\Zar})$.
Since $\sK|_X$ is cofibrant (in our local injective model structure), 
Gillet's construction yields
a map of simplicial presheaves (see \cite[p.~225]{Gillet})
$C_q: \sK|_X \to \sK(\Gamma(q), 2q)|_X$ in $Pre(X_{\Zar})$.
In particular, a map $\sK(X) \to \sK(\Gamma(q), 2q)(X)$.
Furthermore, the naturality of the construction gives for any
morphism $f:Y \to X$ in $\Sm_k$, a  diagram that commutes
up to homotopy
(see, for instance, \cite[5.6.1]{AS}) 
\begin{equation}\label{eqn:Comp-1} 
  \begin{array}{c}
\xymatrix@C1pc{
\sK(X) \ar[r]^<<<{C_q} \ar[d]_{f^*} & \sK(\Gamma(q), 2q)(X) \ar[d]^{f^*} \\
\sK(Y) \ar[r]^<<<{C_q} & \sK(\Gamma(q), 2q)(Y).}
  \end{array}
\end{equation} 
Equivalently, there is a morphism of simplicial presheaves
$C_q: \sK \to \sK(\Gamma(q), 2q)$ in $\sH^{\rm big}_{\rm Zar}(k)$
and hence a morphism in $(Sm_k)_{Nis}$ (see \S~\ref{sec:Notn}).
Pulling back $C_q$ via the morphism of sites
$\pi: (\Sch_k)_{cdh} \to (\Sm_k)_{Nis}$
\cite[p. 111]{Jardine-lh}, and considering the cohomologies
of the associated $cdh$-sheaves, we obtain for any 
$X \in \Sch_k$, a closed subscheme $Z \subseteq X$ and $p, q \ge 0$,
the Chern class maps
\begin{equation}\label{eqn:Comp-2}
\begin{array}{lll}
c^Z_{X, p,q}: \H^{-p}_{Z, cdh}(X, \sK_{cdh}) & := &
\H^{-p}_{Z, cdh}(X, \mathbf L \pi^*(\sK)) \\
& \to & \H^{-p}_{Z, cdh}(X, \mathbf L \pi^*(\sK(\Gamma(q), 2q))) \\
& = & \H^{-p}_{Z, cdh}(X, C_*z_{equi} (\A^q_k,0)_{cdh}) \\
& := & H^{2q-p}_Z(X, \Z(q)).
\end{array}
\end{equation}

It follows from ~\ref{lem:Elem} that
$\H^{-p}_{Z, cdh}(X, \sK_{cdh}) = KH^Z_p(X)$, where
the $KH^Z(X)$ is the homotopy fiber of the map
$KH(X) \to KH(X \setminus Z)$.
Let $(X,Z)$ denote the pair consisting of a scheme $X \in \Sch_k$ and a
closed subscheme $Z \subseteq X$. A map of pairs $f: (Y, W) \to (X,Z)$
is a morphism $f:Y \to X$ such that $f^{-1}(Z) \subseteq W$. 
We have then shown the following.

\begin{thm}\label{thm:Chern-main}
Let $k$ be a field of characteristic zero. Then for any pair
$(X,Z)$ in $\Sch_k$ and for any $p \ge 0, q \in \Z$, there
are Chern class homomorphisms 
\[
c^Z_{X, p,q}: KH^Z_p(X) \to H^{2q-p}_Z(X, \Z(q))
\]
such that the composition of $c^X_{X,0,0}$ with $K_0(X) \to KH_0(X)$ is
the rank map. For any map of pairs
$f: (Y, W) \to (X,Z)$, there is a commutative diagram
\begin{equation}\label{eqn:Comp-2a} 
  \begin{array}{c}
\xymatrix@C2pc{
KH^Z_p(X) \ar[r]^<<<<<{c^Z_{X, p,q}} \ar[d]_{f^*} & 
H^{2q-p}_Z(X, \Z(q)) \ar[d]^{f^*}
\\
KH^W_p(Y) \ar[r]^<<<<<{c^W_{Y, p,q}} & H^{2q-p}_W(Y, \Z(q)).}
  \end{array}
\end{equation}
\end{thm}

\subsection{Chern classes from $KH$-theory to Deligne 
cohomology}\label{sec:Chern-Deligne}
Let $\sC_{\Zar}$ denote the category of schemes which are separated and of finite
type over $\C$ with the Zariski topology. We denote by $\sC_{\Nis}$ the same
category but with the Nisnevich topology.
Let $\sC_{\rm an}$ denote the 
category of complex analytic spaces with the analytic topology.
There is a morphism of sites $\epsilon: \sC_{\rm an} \to \sC_{\Zar}$.
For any $q \in \Z$, let $\Gamma(q)$ denote the complex 
of sheaves on $\sC_{\Zar}$ defined as follows:
\begin{equation}\label{eqn:Deligne-1}
\Gamma(q) = \left\{ \begin{array}{ll}
\Gamma_{\sD}(q) & \mbox{if $q \ge 0$} \\
\mathbf R \epsilon_* ((2\pi \sqrt{-1}) \Z) & \mbox{if $q < 0$},
\end{array}
\right. 
\end{equation}
where $\Gamma_{\sD}(q)$ is the Deligne-Beilinson complex on $\sC_{\Zar}$
in the sense of \cite{EV}. Then $\Gamma(q)$ is a cohomology theory on
$\Sm_{\C}$ satisfying Gillet's conditions for a theory of Chern classes
(see, for instance,  \cite[Ex.~3.4]{AS}).
Applying the argument of ~\ref{thm:Chern-main} in verbatim, we obtain
the Chern class homomorphisms
\begin{equation}\label{eqn:Deligne-2}
c^Z_{X, p,q}: KH^Z_p(X) \to \H^{2q-p}_{Z, cdh}(X, (\Gamma_{\sD}(q))_{cdh})
\end{equation}
for a pair of schemes $(X, Z)$ in $\Sch_{\C}$ which is natural
in $(X,Z)$.

Let us now fix a scheme $X \in \Sch_{\C}$.
Recall from \cite[\S~6.2.5-6.2.8]{DeligneHT} that a smooth proper
hypercovering of $X$ is a smooth simplicial scheme
$X_{\bullet}$ with a map of simplicial schemes $p_X: X_{\bullet} \to X$   
such each map $X_i \to X$ is proper and $p_X$ satisfies the 
universal cohomological descent in the sense of \cite{DeligneHT}.
The resolution of singularities implies that such a hypercovering exists.
The Deligne cohomology of $X$ is defined to be \cite[5.1.11]{DeligneHT}
\begin{equation}\label{eqn:Deligne-2a}
H^p_{\sD}(X, \Z(q)):= \H^p_{\Zar}(X, {\mathbf R} p_{X *} \Gamma_{\sD}(q)) =
\H^p_{\Zar}(X_{\bullet}, \Gamma_{\sD}(q)).
\end{equation}

Gillet's theory of Chern classes gives rise to the Chern class homomorphisms
\begin{equation}\label{eqn:K-Chern}
c^{Q}_{X,p,q}: K_p(X) \to H^{2q-p}_{\sD}(X, \Z(q))
\end{equation}
for any $X \in \Sch_{\C}$ which is contravariant functorial, where
$K_i(X) = \pi_i(\Omega BQP(X))$ is the Quillen $K$-theory
(see, for instance, \cite[\S~2.4]{BPW}).
Our objective is to show that these Chern classes actually factor through
the natural map $K_*(X) \to KH_*(X)$.

The construction of the Chern classes from $KH$-theory to the Deligne
cohomology (see ~\ref{thm:Deligne-Chern-Main} below) will be achieved by
the $cdh$-sheafification of Gillet's Chern classes at the
level of presheaves of simplicial sets, followed by considering the
induced maps on the hypercohomologies. 
Therefore, in order to factor the classical Chern classes 
$c^Q_{X,p,q}$ on Quillen $K$-theory through $KH$-theory,
we only need to identify the target of 
the Chern class maps in ~\eqref{eqn:Deligne-2} with the Deligne
cohomology.

To do this, we let for any $X \in \Sch_{\C}$,
$H^*_{\an}(X, \sF)$ denote the cohomology of the analytic space $X_{\an}$ 
with coefficients in the sheaf $\sF$ on $\sC_{\an}$.
Let $\Z \to Sing^*$ denote a fibrant replacement of the sheaf
$\Z$ on $\sC_{\an}$ so that ${\mathbf R}\epsilon_*(\Z) \xrightarrow{\cong}
\epsilon_*(Sing^*)$. Set $\Z(q) = (2\pi \sqrt{-1})^q\epsilon_*(Sing^*) 
\cong {\mathbf R}\epsilon_*(\Z)$.

\begin{lem}\label{lem:Sing-cdh}
For any $X \in \Sm_{\C}$, the map $H^p_{\an}(X, \Z) \to 
\H^p_{cdh}(X, \Z(q)_{cdh})$ is an isomorphism.
\end{lem}
\begin{proof}
Since $H^p_{\an}(X, \Z) \cong \H^p_{\Zar}(X, \Z(q))$, it suffices to show that
the map $\H^p_{\Zar}(X, \Z(q)) \to \H^p_{cdh}(X, \Z(q)_{cdh})$ is an
isomorphism.

Let $\sC_{\rm loc}$ denote the category of schemes which are separated and of 
finite type over $\C$.  We will consider $\sC_{\rm loc}$ as a Grothendieck site
with coverings given by maps $Y' \to Y$ where the associated map of the analytic
spaces is a local isomorphism of the corresponding topological spaces 
\cite[Expos\'e XI p. 9]{SGAiv}.
Since a Nisnevich cover of schemes is a local isomorphism of the
associated analytic spaces, 
there is a commutative diagram of  morphisms of sites:

\begin{equation}\label{eqn:Sing-cdh-0}
  \begin{array}{c}
\xymatrix@C1pc{
\sC_{\rm loc} \ar[r]^{\delta} \ar[d]_{\nu} & \sC_{\an} \ar[d]^{\epsilon} \\
\sC_{\Nis} \ar[r]^{\tau} & \sC_{\Zar}.}
  \end{array}
\end{equation}

Since every local isomorphism of analytic spaces is refined by open
coverings, it is well known that the map
$\H^p_{\an}(X, \sF^*) \to H^p_{\rm loc}(X, \sF^*)$ is an isomorphism
for any complex of sheaves on $\sC_{\an}$ (see, for instance,
\cite[Prop.~3.3, Thm.~3.12]{Milne}).

We set $(\Z(q))_{Nis} = \tau^*(\Z(q)) = \nu_* \circ \delta^*(Sing^*)$.
We observe that for every $i \in \Z$, the cohomology sheaf
$\sH^i$ associated to the complex $\Z(q)$ is isomorphic to 
the Zariski (or Nisnevich)
sheaf on $\Sch_{\C}$ associated to the presheaf $U \mapsto H^i_{\an}(U, \Z)$.
But this latter presheaf on $\Sm_{\C}$ is homotopy invariant with
transfers. It follows from \cite[Cor.~1.1.1]{SV} that the 
map $\H^p_{\Zar}(X, \Z(q)) \to \H^p_{Nis}(X, (\Z(q))_{Nis})$ is an isomorphism.
We are thus reduced to showing that the map
$\H^p_{Nis}(X, (\Z(q))_{Nis}) \to \H^p_{cdh}(X, (\Z(q))_{cdh})$ is an isomorphism
for $X \in \Sm_{\C}$.

But this follows again from \cite[Cor.~1.1.1, 5.12.3,
Thm.~5.13]{SV} because each $\sH^i \cong
{\mathbf R}^i\nu_*(\Z)$ is a Nisnevich sheaf on $\Sm_{\C}$ associated to the
homotopy invariant presheaf with transfers $U \mapsto H^i_{\an}(U, \Z)$.
The proof is complete.
\end{proof}

For any $X \in \Sch_{\C}$, there are natural maps
\begin{equation}\label{eqn:Deligne-cdh}
H^p_{\sD}(X, \Z(q)) \cong \H^p_{\Zar}(X_{\bullet}, \Gamma_{\sD}(q)) \to
\H^p_{Nis}(X_{\bullet}, (\Gamma_{\sD}(q))_{Nis}) \to 
\H^p_{cdh}(X_{\bullet}, (\Gamma_{\sD}(q))_{cdh}).
\end{equation}

\begin{lem}\label{lem:Deligne-4}
For a projective scheme $X$ over $\C$, the map
$H^p_{\sD}(X, \Z(q)) \to
\H^p_{cdh}(X_{\bullet}, (\Gamma_{\sD}(q))_{cdh})$
is an isomorphism.
\end{lem}
\begin{proof}
Our assumption implies that each component $X_p$ of the simplicial scheme
$X_{\bullet}$ is smooth and projective.
Given a complex of sheaves $\sF^{*}_{\bullet}$ (in Zariski or $cdh$-topology),
there is a spectral sequence
\[
E^{p,q}_1 = \H^q_{{\Zar}/{cdh}}(X_p, (\sF^{*}_p)_{{\Zar}/{cdh}}) \Rightarrow
\H^{p+q}_{{\Zar}/{cdh}}(X_{\bullet}, (\sF^{*}_{\bullet})_{{\Zar}/{cdh}}), 
\]
(see,  for instance, \cite[Appendix]{AS}).
Using this spectral sequence and ~\eqref{eqn:Deligne-cdh}, 
it suffices to show that the map
$H^p_{\Zar}(X, \Gamma_{\sD}(q)) \to \H^p_{cdh}(X, (\Gamma_{\sD}(q))_{cdh})$
is an isomorphism for any smooth projective scheme $X$ over $\C$.
For $q \le 0$, this follows from ~\ref{lem:Sing-cdh}. So we assume $q > 0$.

Since $X$ is smooth and projective, the analytic Deligne complex 
$\Z(q)_{\sD}$ is the complex of analytic sheaves
$\Z(q) \to \sO_{X_{\an}} \to \Omega^1_{X_{\an}} \to \cdots \to 
\Omega^{q-1}_{X_{\an}}$. 
In particular, there is a distinguished triangle
\[
{\mathbf R}\epsilon_*(\Omega^{<q}_{X_{\an}}[-1]) \to \Gamma_{\sD}(q) \to \Z(q) \to
{\mathbf R}\epsilon_*(\Omega^{<q}_{X_{\an}}) \]
in the derived category of sheaves on $X_{\Zar}$.

Since $X$ is projective, it follows from GAGA that the natural map
$\Omega^{<q}_{X/{\C}} \to {\mathbf R}\epsilon_*(\Omega^{<q}_{X_{\an}})$
is an isomorphism in the derived category of sheaves on $X_{\Zar}$. 
In particular, we get a distinguished triangle in the derived category of 
sheaves on $X_{\Zar}$: 
\begin{equation}\label{eqn:Deligne-4-0}
\Omega^{<q}_{X/{\C}}[-1] \to  \Gamma_{\sD}(q) \to \Z(q) \to
\Omega^{<q}_{X/{\C}}. 
\end{equation}
We thus have a commutative diagram of exact sequences:
\begin{align*}
\xymatrix@C.8pc{
\H^{p-1}_{\Zar}(X, \Z(q)) \ar[r] \ar[d] & 
\H^{p-1}_{\Zar}(X, \Omega^{<q}_{X/{\C}}) \ar[r] \ar[d] & 
H^p_{\Zar}(X, \Gamma_{\sD}(q)) \ar[r] \ar[d] & \\
\H^{p-1}_{cdh}(X, (\Z(q))_{cdh}) \ar[r] & 
\H^{p-1}_{cdh}(X, (\Omega^{<q}_{X/{\C}})_{cdh}) \ar[r]  & 
H^p_{cdh}(X, (\Gamma_{\sD}(q))_{cdh}) \ar[r]  &} \\
\xymatrix@C.8pc{ 
\ar[r] & \H^{p}_{\Zar}(X, \Z(q)) \ar[r] \ar[d] &
\H^{p}_{\Zar}(X, \Omega^{<q}_{X/{\C}}) \ar[d] \\
\ar[r] & \H^{p}_{cdh}(X, (\Z(q))_{cdh}) \ar[r]  &
\H^{p}_{cdh}(X, (\Omega^{<q}_{X/{\C}})_{cdh}).}
\end{align*}

It follows from ~\ref{lem:Sing-cdh} that the first and the
fourth vertical arrows from the left are isomorphisms. 
The second and the fifth vertical arrows are isomorphisms by 
\cite[Cor.~2.5]{CHW}.
We conclude that the middle vertical arrow is also an
isomorphism and this completes the proof. 
\end{proof}

As a combination of \ref{lem:Elem}, \eqref{eqn:K-Chern} and ~\ref{lem:Deligne-4}, we
obtain a theory of Chern classes from $KH$-theory to Deligne cohomology
as follows.

\begin{thm}\label{thm:Deligne-Chern-Main}
For every projective scheme $X$ over $\C$, there are Chern class homomorphisms
\[
c_{X, p,q}: KH_p(X) \to H^{2q-p}_{\sD}(X, \Z(q)) 
\]
such that for any morphism of projective $\C$-schemes $f: Y \to X$, 
one has $f^* \circ c_{X, p,q} = c_{Y, p,q} \circ f^*$.
\end{thm}
\begin{proof}
We only need to show that there is a natural isomorphism
$\alpha_X: \H^{p}_{cdh}(X, (\Gamma_{\sD}(q))_{cdh}) \xrightarrow{\cong}
H^{p}_{\sD}(X, \Z(q))$.

Given a morphism of projective $\C$-schemes $f: Y \to X$, there exists a
commutative diagram
\[
\xymatrix@C1pc{
Y_{\bullet} \ar[r]^{f_{\bullet}} \ar[d]_{p_Y} & 
X_{\bullet} \ar[d]^{p_X} \\
Y \ar[r]^{f} & X,}
\]
where the vertical arrows are the simplicial hypercoverings maps.
In particular, there is a commutative diagram
\[
\xymatrix@C.6pc{
\H^p_{\Zar}(X, \Gamma_{\sD}(q)) \ar[rr] \ar[dd] \ar[dr] & & 
\H^p_{\Zar}(Y, \Gamma_{\sD}(q)) \ar[dd] \ar[dr] & \\
&  \H^{p}_{cdh}(X, (\Gamma_{\sD}(q))_{cdh}) \ar[rr] \ar[dd] & & 
\H^{p}_{cdh}(Y, (\Gamma_{\sD}(q))_{cdh}) \ar[dd] \\
H^{p}_{\sD}(X, \Z(q)) \ar[rr] \ar[dr] & &  
H^{p}_{\sD}(Y, \Z(q)) \ar[dr] & \\
&  \H^{p}_{cdh}(X_{\bullet}, (\Gamma_{\sD}(q))_{cdh}) \ar[rr] & &
 \H^{p}_{cdh}(Y_{\bullet}, (\Gamma_{\sD}(q))_{cdh}).}
\]

Using ~\ref{lem:Deligne-4}, we get a map
$\alpha_X: \H^{p}_{cdh}(X, (\Gamma_{\sD}(q))_{cdh}) \to H^{p}_{\sD}(X, \Z(q))$
such that $f^* \circ \alpha_X = \alpha_Y \circ f^*$ for any $f: Y \to X$
as above. Moreover, we have shown in the proof of ~\ref{lem:Deligne-4}
that this map is an isomorphism if $X \in \Sm_{\C}$. 
Since the source as well as the target of $\alpha_X$ 
satisfy $cdh$-descent by ~\ref{lem:Deligne-4}
(see \cite[Lem.~12.1]{SV}), we conclude as in the proof of
~\ref{lem:Elem} that $\alpha_X$ is an isomorphism for every
projective $\C$-scheme $X$.
\end{proof}

\section{Applications II: Intermediate Jacobian and 
Abel-Jacobi map for singular schemes}\label{sec:IJ-AJ}
Recall that a very important object in the study of the geometric
part of motivic cohomology of smooth projective varieties is an intermediate 
Jacobian. The intermediate Jacobians were defined by Griffiths and they
receive the Abel-Jacobi maps from certain subgroups of the geometric
part $H^{2*}(X, \Z(*))$ of the motivic cohomology groups.

A special case of these intermediate
Jacobians is the Albanese variety of a smooth projective variety.
The most celebrated result about the Albanese variety in the context of
algebraic cycles is that the Abel-Jacobi map from the group of 0-cycles
of degree zero to the Albanese variety is an isomorphism on the torsion
subgroups. This theorem of Roitman tells us that the torsion part of the
Chow group of 0-cycles on a smooth projective variety over $\C$ can be
identified with the torsion subgroup of an abelian variety.

Roitman's torsion theorem
has had enormous applications in the theory of algebraic cycles
and algebraic $K$-theory. For example, it was predicted as part of the
conjectures of Bloch and Beilinson that the Chow group of 0-cycles on
smooth affine varieties of dimension at least two should be torsion-free.
This is now a consequence of Roitman's torsion theorem.
We hope to use the Roitman's torsion theorem of this paper to answer 
the analogous question about the motivic cohomology
$H^{2d}(X, \Z(d))$ of a $d$-dimensional singular affine variety
in a future project.

It was predicted as part of the relation between algebraic $K$-theory
and motivic cohomology that the Chow group of 0-cycles should be
(integrally) 
a subgroup of the Grothendieck group. This is also now a consequence of
Roitman's theorem. We shall prove the analogue of this for singular schemes
in the next section. Recall that the Riemann-Roch theorem says that 
this inclusion of the Chow group inside the Grothendieck group is always true
rationally.
For applications concerning the relation between 
Chow groups and {\'e}tale cohomology, see \cite{Bloch-2}.

In this section, we shall apply the theory of Chern classes
from $KH$-theory to Deligne cohomology from \S~\ref{sec:Chern-classes}
to construct the intermediate
Jacobian and Abel-Jacobi map from the geometric part of the
motivic cohomology of any singular projective variety over $\C$.
In the next section, we shall use the Abel-Jacobi map to prove a 
Roitman torsion theorem for singular schemes. 
As another application of our Chern classes and the Roitman torsion theorem,
we shall show that the cycle map from the
geometric part of motivic cohomology to the $KH$ groups, constructed in 
~\ref{thm:Conn-KGL}, is injective for a large class of schemes.

\subsection{The Abel-Jacobi map}\label{sec:AJ}
In the rest of this section, we shall consider all schemes over $\C$
and mostly deal with projective schemes. Let $X$ be a projective scheme
over $\C$ of dimension $d$. Let $X_{\rm sing}$ and $X_{\rm reg}$ denote
the singular (with the reduced induced subscheme structure) and the smooth loci
of $X$, respectively. Let $r$ denote the number of $d$-dimensional 
irreducible components of $X$. We shall fix a resolution of singularities 
$f: \wt{X} \to X$ and let $E = f^{-1}(X_{\rm sing})$ throughout this section.
The following is an immediate consequence of the $cdh$-descent for 
Deligne cohomology.

\begin{lem}\label{lem:vanish-Deligne}
For any integer $q \ge d+1$, one has $H^{q+d+i}_{\sD}(X, \Z(q)) = 0$
for $i \ge 1$.
\end{lem}
\begin{proof}
If $X$ is smooth, it follows immediately from ~\eqref{eqn:Deligne-4-0}.
In general, the $cdh$-descent for Deligne cohomology 
(see ~\ref{lem:Deligne-4} or \cite[Variant~3.2]{BPW}) implies that
there is an exact sequence
\[
H^{q+d+i-1}_{\sD}(E, \Z(q)) \to H^{q+d+i}_{\sD}(X, \Z(q)) \to
H^{q+d+i}_{\sD}(\wt{X}, \Z(q)) \oplus H^{q+d+i}_{\sD}(X_{\rm sing}, \Z(q)).
\]
We conclude the proof by using this exact sequence and induction on
$\dim(X)$.
\end{proof}
  
It follows from the definition of the
Deligne cohomology that there is a natural map of complexes
${\Gamma_{\sD}(q)}|_X \to {\Z(q)}|_X$ (see ~\eqref{eqn:Deligne-4-0}) and in 
particular, there is a natural map
$H^p_{\sD}(X, \Z(q)) \xrightarrow{\kappa_X} H^p_{\an}(X, \Z(q))$.
For any integer $0 \le q \le d$, the {\sl intermediate Jacobian}
$J^q(X)$ is defined so that we have an exact sequence
\[
0 \to J^q(X) \to H^{2q}_{\sD}(X, \Z(q)) \xrightarrow{\kappa_X} 
H^{2q}_{\an}(X, \Z(q)).
\]

It follows from ~\ref{thm:Deligne-Chern-Main} that there is 
a commutative diagram
\begin{equation}\label{eqn:sing-iso-top}
  \begin{array}{c}
\xymatrix@C1pc{
KH_0(X) \ar[r]^-{c_{X,d,0}} \ar[d]_{f^*} &
H^{2d}_{\sD}(X, \Z(d)) \ar[r]^-{\kappa_X} \ar[d]^{f^*} & 
H^{2d}_{\an}(X, \Z(d)) \ar[d]^{f^*} \\
KH_0(\wt{X}) \ar[r]^-{c_{X,d,0}} & 
H^{2d}_{\sD}(\wt{X}, \Z(d)) \ar[r]^-{\kappa_{\wt{X}}} & 
H^{2d}_{\an}(\wt{X}, \Z(d)).}
  \end{array}
\end{equation}

It follows from ~\eqref{eqn:Deligne-4-0} that $\kappa_{\wt{X}}$ is surjective.
The $cdh$-descent for the Deligne cohomology and ~\ref{lem:vanish-Deligne}
together imply that the middle vertical arrow in ~\eqref{eqn:sing-iso-top}
is surjective.
The cdh-excision property of  singular cohomology 
(see \cite[8.3.10]{DeligneHT}) yields an exact sequence
\[
H^{2d-1}_{\an}(E, \Z(d)) \to H^{2d}_{\an}(X, \Z(d)) \to
H^{2d}_{\an}(\wt{X}, \Z(d)) \oplus H^{2d}_{\an}(X_{\rm sing}, \Z(d)) \to
H^{2d+1}_{\an}(E, \Z(d)).
\]
Since $X_{\rm sing}$ and $E$ are projective schemes of dimension at most
$d-1$, it follows that the right vertical arrow in ~\eqref{eqn:sing-iso-top}
is an isomorphism. We conclude that there is a short exact sequence
\begin{equation}\label{eqn:IJ-0}
0 \to J^d(X) \to H^{2d}_{\sD}(X, \Z(d)) \xrightarrow{\kappa_X} 
H^{2d}_{\an}(X, \Z(d)) \to 0.
\end{equation}

A similar Mayer-Vietoris property of the motivic cohomology yields an
exact sequence
\[
H^{2d-1}(E, \Z(d)) \to H^{2d}(X, \Z(d)) \to
H^{2d}(\wt{X}, \Z(d)) \oplus H^{2d}(X_{\rm sing}, \Z(d)) \to
H^{2d+1}(E, \Z(d)).
\]
It follows from ~\ref{thm:MGL-MC-iso} that
$H^{2d}(X_{\rm sing}, \Z(d)) = H^{2d+1}(E, \Z(d)) = 0$.
In particular, there exists a short exact sequence
\begin{equation}\label{eqn:IJ-1}
0 \to \frac{H^{2d-1}(E, \Z(d))}{H^{2d-1}(\wt{X}, \Z(d)) +
H^{2d-1}(X_{\rm sing}, \Z(d))} \to
H^{2d}(X, \Z(d)) \to  H^{2d}(\wt{X}, \Z(d)) \to 0.
\end{equation}

Since the map $H^{2d}(\wt{X}, \Z(d)) \cong \CH^d(\wt{X}) \to 
H^{2d}_{\an}(\wt{X}, \Z(d))$ is the degree map which is surjective,
we conclude that the `degree' map
$H^{2d}(X, \Z(d)) \to H^{2d}_{\an}(X, \Z(d))$ is also surjective.
We let $A^d(X)$ denote the kernel of this degree map.

It follows from ~\ref{thm:Deligne-Chern-Main} that there
is a Chern class map (take $p =0$) $c_{X,q}: KH_0(X) \to H^{2q}_{\sD}(X, \Z(q))$.
Theorem~\ref{thm:Conn-KGL} says that the spectral sequence
~\eqref{eqn:KH-S-0} induces a cycle class map
$cyc_{X, 0}: H^{2d}(X, \Z(d)) \to KH_0(X)$. Composing the two maps,
we get a cycle class map from motivic to Deligne cohomology 
\begin{equation}\label{eqn:IJ-2}
\wt{c}^d_X: H^{2d}(X, \Z(d)) \to H^{2q}_{\sD}(X, \Z(q))
\end{equation}
and a commutative diagram of short exact sequences:
\begin{equation}\label{eqn:IJ-3}
  \begin{array}{c}
\xymatrix@C.8pc{
0 \ar[r] & A^d(X) \ar[d]_{{\rm AJ}^d_X} \ar[r] & H^{2d}(X, \Z(d)) \ar[r] 
\ar[d]^{\wt{c}^d_X} &
H^{2d}_{\an}(X, \Z(d)) \ar@{=}[d] \ar[r] & 0 \\
0 \ar[r] & J^d(X) \ar[r] & H^{2q}_{\sD}(X, \Z(q)) \ar[r] & 
H^{2d}_{\an}(X, \Z(d)) \ar[r] & 0.}
  \end{array}
\end{equation}

It is known that $J^d(X)$ is a semi-abelian variety whose abelian
variety quotient is the classical Albanese variety of $\wt{X}$
(see \cite[Thm.~1.1]{BS} or \cite{BAS}). The induced map
${\rm AJ}^d_X: A^d(X) \to J^d(X)$ will be called the
{\sl Abel-Jacobi} map for the singular scheme $X$.
We shall prove our main result about this Abel-Jacobi map in the next
section. Here, we recall the following description of $J^d(X)$ in terms
of 1-motives. Recall from \cite[\S~12.12]{BKahn} that every projective
scheme $X$ of dimension $d$ over $\C$ has an 1-motive ${\Alb}^{+}(X)$
associated to it. This is called the cohomological Albanese $1$-motive
of $X$. This is a generalization of the Albanese variety of smooth projective
schemes.

\begin{thm}$($\cite[Cor.~3.3.2]{BAS}$)$\label{thm:Alb-J}
For a projective scheme $X$ of dimension $d$ over $\C$, there is a 
canonical isomorphism $J^d(X) \cong {\rm Alb}^{+}(X)$.
\end{thm}

\subsection{Levine-Weibel Chow group and motivic cohomology}
\label{sec:LW-mot}
In order to prove our main theorem of this section, we
need to compare the motivic cohomology of singular schemes
with another `motivic cohomology', called the (cohomological) 
Chow-group of 0-cycles, introduced by Levine and Weibel \cite{LW}.
We shall assume throughout our discussion that $X$ is a reduced projective
scheme of dimension $d$ over $\C$. However, we remark that the
following definition of the Chow group of 0-cycles makes sense over any
ground field. 
Let $\sZ_0(X)$ denote the free abelian group on the closed points of 
$X_{\rm reg}$.

\begin{defn}\label{defn:0-cycle-S-1}
Let $C$ be a pure dimension one reduced scheme in $\Sch_{\C}$. 
We shall say that a pair $(C, Z)$ is \emph{a good curve
relative to $X$} if there exists a finite morphism $\nu\colon C \to X$
and a  closed proper subscheme $Z \subsetneq C$ such that the following hold.
\begin{enumerate}
\item
No component of $C$ is contained in $Z$.
\item
$\nu^{-1}(X_{\rm sing}) \cup C_{\rm sing}\subseteq Z$.
\item
$\nu$ is local complete intersection morphism  at every 
point $x \in C$ such that $\nu(x) \in X_{\rm sing}$. 
\end{enumerate}
\end{defn}

Let $(C, Z)$ be a good curve relative to $X$ and let 
$\{\eta_1, \cdots , \eta_r\}$ be the set of generic points of $C$. 
Let $\cO_{C,Z}$ denote the semilocal ring of $C$ at 
$S = Z \cup \{\eta_1, \cdots , \eta_r\}$.
Let $\C(C)$ denote the ring of total
quotients of $C$ and write $\cO_{C,Z}^\times$ for the group of units in 
$\cO_{C,Z}$. Notice that $\cO_{C,Z}$ coincides with $k(C)$ 
if $|Z| = \emptyset$. 
As $C$ is Cohen-Macaulay, $\cO_{C,Z}^\times$  is the subgroup of $k(C)^\times$ 
consisting of those $f$ which are regular and invertible in the local rings 
$\cO_{C,x}$ for every $x\in Z$.

Given any $f \in \sO^{\times}_{C, Z} \inj \C(C)^{\times}$, we denote by  
${\rm div}(f)$
 the divisor of zeros and poles of $f$ on $C$, that is defined as follows. If 
$C_1,\ldots, C_r$ are the irreducible components of $C$, we set 
${\rm div}(f)$ to be the $0$-cycle $\sum_{i=1}^r {\rm div}(f_i)$, where 
$(f_1, \cdots , f_r)  = \theta_{(C,Z)}(f)$ and ${\rm div}(f_i)$ is the usual 
divisor of a rational function on an integral curve in the sense of
\cite{Fulton}. Let $\sZ_0(C, Z)$ denote the free abelian group on the
closed points of $C \setminus Z$.
As $f$ is an invertible 
regular function on $C$ along $Z$, ${\rm div}(f)\in \sZ_0(C,Z)$.

By definition, given any good curve $(C,Z)$ relative to $X$, we have a 
pushforward map $\sZ_0(C,Z)\xrightarrow{\nu_{*}} \sZ_0(X)$.
We shall write $\sR_0(C, Z, X)$ for the subgroup
of $\sZ_0(X)$ generated by the set 
$\{\nu_*({\rm div}(f))| f \in \sO^{\times}_{C, Z}\}$. 
Let $\sR^{BK}_0(X)$ denote the subgroup of $\sZ_0(X)$ generated by 
the image of the map $\sZ_0(C, Z, X) \to \sZ_0(X)$, where
$\sZ_0(C, Z, X)$ runs through all good curves.
We let $\CH^{BK}_0(X) = \frac{\sZ_0(X)}{\sR^{BK}_0(X)}$.

If we let $\sR^{LW}_0(X)$ denote the subgroup of $\sZ_0(X)$ generated
by the divisors of rational functions on good curves as above, where
we further assume that the map $\nu: C \to X$ is a closed immersion,
then the resulting quotient group ${\sZ_0(X)}/{\sR^{LW}_0(X)}$ is
denoted by $\CH^{LW}_0(X)$. There is a canonical surjection
$\CH^{LW}_0(X) \surj \CH^{BK}_0(X)$. However, we can say more about this 
map in the present context. This comparison will be an essential ingredient in
the proof of ~\ref{thm:Roitman-sing}.

\begin{thm}\label{thm:BK-com}
For a projective scheme $X$ over $\C$, the map 
$\CH^{LW}_0(X) \surj \CH^{BK}_0(X)$ is an isomorphism.
\end{thm}
\begin{proof}
By \cite[Lem.~3.13]{BK-1}, there are cycle class maps
$\CH^{LW}_0(X) \surj \CH^{BK}_0(X) \to K_0(X)$ and one knows from 
\cite[Cor.~2.7]{Levine} that the kernel of the composite map is
$(d-1)!$-torsion. In follows that ${\rm Ker}(\CH^{LW}_0(X) \to \CH^{BK}_0(X))$
is torsion. In particular, it lies in $\CH^{LW}_0(X)_{{\rm deg} \ 0}$.

On the other hand, it follows from \cite[Prop.~9.7]{BK-1} that
the Abel-Jacobi map $\CH^{LW}_0(X)_{{\rm deg} \ 0} \to J^d(X)$ 
(see \cite[Th.~1.1]{BS}) factors through 
$\CH^{LW}_0(X)_{{\rm deg} \ 0} \surj \CH^{BK}_0(X)_{{\rm deg} \ 0} \to J^d(X)$.
Moreover, it follows from \cite[Th.~1.1]{BS} that the
composite map is isomorphism on the torsion subgroups.
In particular, ${\rm Ker}(\CH^{LW}_0(X)_{{\rm deg} \ 0} \surj 
\CH^{BK}_0(X)_{{\rm deg} \ 0})$ is torsion-free. It must therefore be zero. 
\end{proof}

In the rest of this text, we shall identify the above two Chow groups
for projective schemes over $\C$ and write them as $\CH^d(X)$.
There is a degree map $deg_X: \CH^d(X) \to  H^{2d}_{\an}(X, \Z(d))
\cong \Z^{r}$. Let $\CH^d(X)_{{\rm deg} \ 0}$
denote the kernel of this degree map.
In order to obtain applications of the above Abel-Jacobi map,
we connect $\CH^d(X)$ with the motivic cohomology as follows.

\begin{lem}\label{lem:Comp-Chow}
There is a canonical map $\gamma_X: \CH^d(X) \to H^{2d}(X, \Z(d))$ 
which restricts to a map $\gamma_X: \CH^d(X)_{{\rm deg} \ 0} \to A^d(X)$.
\end{lem}
\begin{proof}
We let $U$ denote the smooth locus of $X$ and let $x \in U$ be a closed point. 
The excision for the local cohomology
with support in a closed subscheme tells us that the map 
$\H^{0}_{\{x\}, cdh}(X, {C_* z_{equi}(\A^d_{\C}, 0)}_{cdh}) \to
\H^{0}_{\{x\}, cdh}(U, {C_* z_{equi}(\A^d_{\C}, 0)}_{cdh})$ is an isomorphism.
On the other hand, the purity theorem for the motivic cohomology of 
smooth schemes and the isomorphism between the motivic cohomology and
higher Chow groups \cite{Voev-1} imply that the map 
$\H^{0}_{\{x\}, cdh}(U, {C_* z_{equi}(\A^d_{\C}, 0)}_{cdh})
\to \H^{0}_{cdh}(U, {C_* z_{equi}(\A^d_{\C}, 0)}_{cdh})$
is same as the map of the Chow groups $\Z \cong \CH_0(\{x\}) \to
\CH_0(U)$. 
In particular, we obtain a map
\[
\gamma_x: \Z \to \H^{0}_{\{x\}, cdh}(X, {C_* z_{equi}(\A^d_{\C}, 0)}_{cdh}) \to
\H^{0}_{cdh}(X, {C_* z_{equi}(\A^d_{\C}, 0)}_{cdh}) = 
H^{2d}(X, \Z(d)). 
\]

We let $\gamma_X([x])$ be the image of
$1 \in \Z$ under this map. This yields a homomorphism
$\gamma_X: \sZ_0(X) \to H^{2d}(X, \Z(d))$.
We now show that this map kills $\sR_0(X)$.

We first assume that $X$ is a reduced curve. In this case, an easy
application of the spectral sequence of ~\ref{thm:KH-SS} and the 
vanishing result of ~\ref{thm:MGL-MC-iso} shows that there is a
short exact sequence 
\begin{equation}\label{eqn:Comp-Chow-0}
0 \to H^2(X, \Z(1)) \to KH_0(X) \to H^0(X, \Z(0)) \to 0.
\end{equation}

Using the isomorphism $H^0(X, \Z(0)) \xrightarrow{\cong} 
H^{0}_{\an}(X, \Z)$ and the natural map $K_*(X) \to KH_*(X)$,
we have a commutative diagram of the short exact sequences
\begin{equation}\label{eqn:Comp-Chow-1}
  \begin{array}{c}
\xymatrix@C.8pc{
0 \ar[r] & \Pic(X) \ar[r] & K_0(X) \ar[r] \ar[d] & H^{0}_{\an}(X, \Z) \ar[d] 
\ar[r] & 0 \\
0 \ar[r] & H^2(X, \Z(1)) \ar[r] & KH_0(X) \ar[r] & H^{0}_{\an}(X, \Z)  
\ar[r] & 0.}
  \end{array}
\end{equation}

It follows from \cite[Lem.~3.11]{BK-1} that the map
$\sZ_0(X) \to K_0(X)$ given by $cyc_X([x]) = [\sO_{\{x\}}] \in K_0(X)$
defines an isomorphism $\CH^1(X) \xrightarrow{\cong} \Pic(X)$.
Note that $x \in U$ and hence the class $[\sO_{\{x\}}]$ in $K_0(X)$ makes sense. 
We conclude from this isomorphism and ~\eqref{eqn:Comp-Chow-1} that
the composite map $\sZ_0(X) \to K_0(X) \to KH_0(X)$ has image
in $H^2(X, \Z(1))$ and it factors through $\CH^1(X)$.

We now assume $d \ge 2$ and let $\nu:(C,Z) \to X$ be a good curve and
let $f \in \sO^{\times}_{C,Z}$. We need to show that
$\gamma_X(\nu_*(\divf(f))) = 0$.
By \cite[Lem.~3.4]{BK-1}, we can assume 
that $\nu$ is an lci morphism. 
In particular, there is a functorial
push-forward map $\nu_*: H^2(C, \Z(1)) \to H^{2d}(X, \Z(d))$
by ~\ref{thm:cdhdescent-2} and \cite[Def.~2.32, Thm.~2.33]{Navarro}.
We thus have a commutative diagram:
\begin{equation}\label{eqn:Comp-Chow-2}
  \begin{array}{c}
\xymatrix@C.8pc{
\sZ_0(C,Z) \ar@/^1pc/[rr]^{\gamma_C} \ar[r]_-{\cong} \ar[d]_{\nu_*} & 
\oplus_{x \notin Z} H^0(\{x\}, \Z(0)) \ar[r] \ar[d]^{\nu_*} & 
H^2(C, \Z(1)) \ar[d]^{\nu_*} \\
\sZ_0(X) \ar[r]^-{\cong} \ar@/_1pc/[rr]_{\gamma_X} & 
\oplus_{x \notin X_{\rm sing}} H^0(\{x\}, \Z(0))
\ar[r] & H^{2d}(X, \Z(d)).}
  \end{array}
\end{equation}

The two horizontal arrows on the right are the push-forward maps
on the motivic cohomology since the inclusion $\{x\} \inj X$
is an lci morphism for every $x \notin X_{\rm sing}$.
We have shown that $\gamma_C(\divf(f)) = 0$ and hence
$\gamma_X(\nu_*(\divf(f))) = \nu_*(\gamma_C((\divf(f)))) = 0$.
Furthermore, the composite $\sZ_0(X) \to H^{2d}(X, \Z(d)) \to
H^{2d}(\wt{X}, \Z(d)) \to H^{2d}_{\an}(\wt{X}, \Z(d)) \cong \Z^r$ is
the degree map. This shows that $\gamma_X(\sZ_0(X)_{{\rm deg} \ 0})
\subseteq A^d(X)$.
\end{proof}

\section{Applications III: Roitman torsion and cycle class 
map}\label{sec:Roit}
Let $X$ be a projective scheme of dimension $d$ over $\C$.
Using the map $\gamma_X: \CH^d(X) \to H^{2d}(X, \Z(d))$ and the Abel-Jacobi
map ${\rm AJ}^d_X$ of ~\eqref{eqn:IJ-3}, we shall now prove our main
result on the Abel-Jacobi map and Roitman torsion for singular schemes. 
We shall use the following lemma in the proof.

\begin{lem}\label{lem:K-KH}
Let $X$ be a reduced projective scheme of dimension $d$ over $\C$.
There is a cycle class map $cyc^Q_{X,0}: \CH^d(X) \to K_0(X)$
and a commutative diagram
\begin{equation}\label{eqn:K-KH-0}
  \begin{array}{c}
\xymatrix@C1pc{
\CH^d(X) \ar[r]^{cyc^Q_{X,0}} \ar[d]_{\gamma_X} & K_0(X) \ar[d] \\
H^{2d}(X, \Z(d)) \ar[r]^{cyc_{X,0}} & KH_0(X).}
  \end{array}
\end{equation}
\end{lem}
\begin{proof}
Every closed point $x \in U$ defines the natural map
$\Z = K_0(\{x\}) = K^{\{x\}}_0(X) \to K_0(X)$ and hence 
a class $[\sO_{\{x\}}] \in K_0(X)$.
This defines a map $cyc^Q_{X,0}: \sZ_0(X) \to K_0(X)$ and it 
factors through $ \CH^d(X)$ by \cite[Prop.~2.1]{LW}.  
Since $\CH^d(X)$ is generated by the closed points in $U$, 
it suffices to show that for every closed point $x \in U$, the diagram
\begin{equation}\label{eqn:K-KH-0-1}
  \begin{array}{c}
\xymatrix@C1pc{
K^{\{x\}}_0(X) \ar[r] \ar@{=}[d] & K_0(X) \ar[d] \\
KH^{\{x\}}_0(X) \ar[r] & KH_0(X)}
  \end{array}
\end{equation}
commutes. But this is clear from the functorial properties of the map
of presheaves $K(-) \to KH(-)$ on $\Sch_{\C}$.
\end{proof}

We can now prove:

\begin{thm}\label{thm:Roitman-sing}
Let $X$ be a projective scheme over $\C$ of dimension $d$. 
Assume that either $d \le 2$ or $X$ is regular in codimension one.
Then, there is 
a semi-abelian variety $J^d(X)$ and an Abel-Jacobi map 
${\rm AJ}^d_X: A^d(X) \to J^d(X)$ which is surjective and whose restriction to the
torsion subgroups
${\rm AJ}^d_X: A^d(X)_{\rm tors} \to J^d(X)_{\rm tors}$ is an isomorphism.
\end{thm}
\begin{proof}
We can assume that $X$ is reduced. We first consider the case when $X$ has 
dimension at most two but has arbitrary singularity.
In this case, we only need to prove that ${\rm AJ}^d_X$ is surjective and its 
restriction to the torsion subgroups is an isomorphism.

The map ${\rm AJ}^d_X$ is induced by the Chern class map
$c_{X,d,0}: KH_0(X) \to H^{2d}_{\sD}(X, \Z(d))$ and the composite
map $K_0(X) \to KH_0(X) \to H^{2d}_{\sD}(X, \Z(d))$ is the Gillet's
Chern class map $C^Q_{X,d,0}$ of ~\eqref{eqn:K-Chern}.
Composing these maps with the cycle class maps and using ~\ref{lem:K-KH},
we get a commutative diagram
\begin{equation}\label{eqn:Roitman-Main-0}
  \begin{array}{c}
\xymatrix@C1pc{
\CH^d(X)_{{\rm deg} \ 0} \ar[r]^>>>{\gamma_X} \ar[dr]_{{\rm AJ}^{d,Q}_X} & 
A^d(X) \ar[d]^{{\rm AJ}^d_X} \\
& J^d(X).}
  \end{array}
\end{equation}


The map ${\rm AJ}^{d,Q}_X$ is surjective and is an isomorphism on the
torsion subgroups by \cite[Main Theorem]{BPW}. 
It follows that ${\rm AJ}^d_X$ is also surjective.
To prove that it is an isomorphism on the torsion subgroups, we
apply ~\ref{thm:Alb-J} and \cite[Cor.~13.7.5]{BKahn}. 
It follows from these results that there is indeed an isomorphism  
$\phi^d_X: {J^d(X)}_{\rm tor} \xrightarrow{\cong} {A^d(X)}_{\rm tor}$.
Since $J^d(X)$ is a semi-abelian variety, we know that for any given
integer $n \ge 1$, the $n$-torsion subgroup ${_nJ^d(X)}$ is finite.
It follows that ${_nA^d(X)}$ and ${_nJ^d(X)}$ are finite abelian groups of
same order.
We conclude that the Abel-Jacobi map ${\rm AJ}^d_X: A^d(X) \to J^d(X)$
induces the map ${\rm AJ}^d_X: {_nA^d(X)} \to {_nJ^d(X)}$ between
finite abelian groups which have same order.
Therefore, this map will be an isomorphism if and only if it is a surjection.
But this is true by ~\eqref{eqn:Roitman-Main-0} because we have seen above
that the composite map ${\rm AJ}^{d,Q}_X$
is isomorphism between the $n$-torsion subgroups.
Since $n \ge 1$ is arbitrary in this argument, we conclude the proof of the
theorem.

We now consider the case when
$X$ has arbitrary dimension but it is regular in codimension
one. Let $f: \wt{X} \to X$ be a resolution of singularities of $X$.
It is then known that $J^d(X) \cong J^d(\wt{X}) =
{\rm Alb}(\wt{X})$
(see \cite[Rem.~2, pp.~505]{Mallick}). We have a commutative diagram
\begin{equation}\label{eqn:Normal-Roit**}
\xymatrix@C1pc{
\CH^d(X)_{{\rm deg} \ 0} \ar[rr]^-{\gamma_X} 
\ar[drr]_{{\rm AJ}^{LW}_X} &  
\ar[r] & A^d(X) \ar@{->>}[r]^-{f^*} \ar@{..>}[d]^{{\rm AJ}^d_X} 
& A^d(\wt{X}) \ar@{->>}[d]^{{\rm AJ}^d_{\wt{X}}} \\
& & J^d(X) \ar[r]^-{\cong} & J^d(\wt{X}).}
\end{equation}

Since the lower horizontal arrow in this diagram is an isomorphism, it
uniquely defines the Abel-Jacobi map ${\rm AJ}^d_X$.
The map $f^*\circ \gamma_X$ is known to be 
surjective by the moving lemma for 0-cycles on smooth schemes.
In particular, $f^*$ is surjective. The map ${\rm AJ}^d_{\wt{X}}$ is
also known to be surjective. It follows that ${\rm AJ}^d_X$ is surjective.

To prove that this is an isomorphism on the torsion subgroups,
we can argue exactly as in the first case of the theorem.
This reduces us to showing that ${\rm AJ}^d_X$ is surjective
on the $n$-torsion subgroups for every given integer $n \ge 1$.
But this follows because ${\rm AJ}^{LW}_X$ (and also ${\rm AJ}^d_{\wt{X}}$)
is isomorphism
on the $n$-torsion subgroups by \cite[Th.~1.1]{BS}. This finishes the proof
of the theorem. 
\end{proof}

\begin{remk}\label{remk:higher-dim}
For arbitrary $d \ge 1$, the map ${\rm AJ}^{d,Q}_X$ in 
~\eqref{eqn:Roitman-Main-0} is known to be an isomorphism only up to 
multiplication by $(d-1)!$. This prevents us from extending 
~\ref{thm:Roitman-sing}  to 
higher dimensions if $X$ has singularities in codimension one.
We also warn the reader that unlike ${\rm AJ}^{d,Q}_X$
in ~\eqref{eqn:Roitman-Main-0}, the map ${\rm AJ}^{LW}_X$ in
~\eqref{eqn:Normal-Roit**} is not
defined via the Chern class map on $K_0(X)$. These maps coincide only up to
multiplication by $(d-1)!$.
\end{remk}

\subsection{Injectivity of the cycle class map}\label{sec:ICCM}
Like the case of smooth schemes, the Roitman torsion theorem for 
singular schemes has many potential applications. Here, we
use this to prove our next main result of this section.
It was shown by Levine in \cite[Thm. 3.2]{Levine}
that for a smooth projective scheme $X$ of dimension $d$ over $\C$, 
the cycle class map $H^{2d}(X, \Z(d)) \to K_0(X)$ (see 
~\eqref{eqn:Conn-cycle}) is injective.
We generalize this to singular schemes as follows.

\begin{thm}\label{thm:Cycle-class-inj}
Let $X$ be a projective scheme of dimension $d$ over $\C$.
Assume that either $d \le 2$ or $X$ is regular in codimension one.
Then the cycle class map $cyc_0: H^{2d}(X, \Z(d)) \to KH_0(X)$ is injective.
\end{thm}
\begin{proof}
We note that $cyc_0:H^{2d}(X, \Z(d)) \to KH_0(X)$ is induced 
by the spectral sequences ~\eqref{eqn:KH-S-0} and ~\eqref{eqn:Conn-cycle},
both of which degenerate with rational coefficients.
In particular, ${\rm Ker}(cyc_0)$ is a torsion group.

On the other hand, if $\dim(X) \le 2$, ~\eqref{eqn:IJ-3} and 
~\ref{thm:Roitman-sing} tell us that the composite map
$\wt{c}^d_X: H^{2d}(X, \Z(d)) \xrightarrow{cyc_0} KH_0(X) \xrightarrow{c_{X,0,d}}
H^{2d}_{\sD}(X, \Z(d))$  is isomorphism on the torsion subgroups.
We must therefore have ${\rm Ker}(cyc_0) = 0$.

If $X$ is regular in codimension one, we let $\wt{X} \to X$ be a 
resolution of singularities and consider the commutative diagram
\[
\xymatrix@C1pc{
H^{2d}(X, \Z(d)) \ar[r]^-{cyc_X,0} \ar[d]_{f^*} & KH_0(X) \ar[d]^{f^*} \\
H^{2d}(\wt{X}, \Z(d)) \ar[r]^-{cyc_{\wt{X}, 0}} & K_0(\wt{X}).}
\]
We have shown in the proof of ~\ref{thm:Roitman-sing} that the
left vertical arrow is isomorphism on the torsion subgroups.
The bottom horizontal arrow is injective by \cite[Th.~3.2]{Levine}.
It follows that $cyc_{X,0}$ is injective on the torsion
subgroup.  We must therefore have ${\rm Ker}(cyc_{X,0}) = 0$.
This finishes the proof.
\end{proof}

\noindent\emph{Acknowledgements.} 
The authors would like to thank Shane Kelly for having read an earlier version
of this paper and pointing out some missing arguments. The authors would like
to thank the referee for very carefully reading the paper and suggesting
many improvements.

\end{document}